\documentclass[10pt]{article}
\usepackage[ansinew]{inputenc}
\usepackage{amsmath}
\usepackage{amsfonts}
\usepackage{amssymb}
\usepackage{amsthm}
\usepackage{newlfont}
\usepackage{amsbsy}
\usepackage{bm}
\usepackage{color}
\usepackage{epsfig}
\usepackage{dsfont}
\usepackage{latexsym}
\usepackage{array}
\usepackage{longtable}
\usepackage{enumerate}
\usepackage{multicol}
\usepackage{mathrsfs}
\usepackage{wasysym}
\usepackage{graphics}
\usepackage{graphicx}
\usepackage{hyperref}

\newtheorem{lema}{Lemma}[section]
\newtheorem{teor}{Theorem}[section]
\newtheorem{corol}{Corollary}[section]

\newtheorem{prop}{Proposition}[section]
\newtheorem{rem}{\it Remark}[section]
\newcommand{\nor}[2]{{\left\|{#1}\right\|_{#2}}}
\newcommand{\nora}[3]{{\left\|{#1}\right\|_{#2}^{#3}}}

\title{\textsc{Well-posedness for a Family of \\ Perturbations of the KdV Equation in \\ Periodic Sobolev Spaces of Negative Order}}
\author{Xavier Carvajal Paredes\thanks{Xavier Carvajal. Instituto de Matem\'aticas, Universidade Federal do Rio de Janeiro, Rio de Janeiro. E-mail:  {\tt carvajal@im.ufrj.br}} \\ Ricardo A. Pastran  \thanks{Ricardo Pastr\'an. Instituto de Matem\'aticas, Universidade Federal do Rio de Janeiro, Rio de Janeiro. On leave from Universidad Nacional de Colombia, Bogot\'a.
E-mail: {\tt rapastranr@unal.edu.co}}}
\begin{document}
\maketitle
%\vspace{1cm}
\begin{abstract}
We establish local well-posedness in Sobolev spaces $H^s(\mathbb{T})$, with $s\geq -1/2$, for the initial value problem issues of the equation
$$ u_t + u_{xxx}+\eta Lu + uu_x=0;\; x\in \mathbb{T},\; t\geq0, $$
where $\eta >0$, $(Lu)^{\wedge }(k)=-\Phi(k)\widehat{u}(k)$, $k\in \mathbb{Z}$ and $\Phi \in \mathbb{R}$ is bounded above. Particular cases of this problem are the Korteweg-de Vries-Burgers equation for $\Phi(k)=-k^2$, the derivative Korteweg-de Vries-Kuramoto-Sivashinsky equation for $\Phi(k)=k^2-k^4$, and the Ostrovsky-Stepanyams-Tsimring equation for $\Phi(k)=|k|-|k|^3$.
\end{abstract}
%\frontmatter
\vspace{0.5cm}
\textit{Keywords:} Cauchy Problem, Local Well-Posedness, KdV equation.
%%%%%%%%%%%%%%%%%%%%%%%%%%%%%%%%%%%%%%%%%%%%%%%%%%%%%%%%%%%%%%%%%%%%%%%%%%%%%%%%%%%%%%%%%%%%%%%%%%%%%%%%%%%%%%%%%%%%%%%%%%%%%%%%%%%%%%%%%%%%%%%%%%%%%%%%%
%%%%%%%%%%%%%%%%%%%%%%%%%%%%%%%%%%%%%%%%%%%%%%%%%%%%%%%%%%%%%%%%%%%%%%%%%%%%%%%%%%%%%%%%%%%%%%%%%%%%%%%%%%%%%%%%%%%%%%%%%%%%%%%%%%%%%%%%%%%%%%%%%%%%%%%%%
%%%%%%%%%%%%%%%%%%%%%%%%%%%%%%%%%%%%%%%%%%%%%%%%%%%%%%%%%%%%%%%%%%%%%%%%%%%%%%%%%%%%%%%%%%%%%%%%%%%%%%%%%%%%%%%%%%%%%%%%%%%%%%%%%%%%%%%%%%%%%%%%%%%%%%%%%
%%%%%%%%%%%%%%%%%%%%%%%%%%%%%%%%%%%%%%%%%%%%%%%%%%%%%%%%%%%%%%%%%%%%%%%%%%%%%%%%%%%%%%%%%%%%%%%%%%%%%%%%%%%%%%%%%%%%%%%%%%%%%%%%%%%%%%%%%%%%%%%%%%%%%%%%%
%%%%%%%%%%%%%%%%%%%%%%%%%%%%%%%%%%%%%%%%%%%%%%%%%%%%%%%%%%%%%%%%%%%%%%%%%%%%%%%%%%%%%%%%%%%%%%%%%%%%%%%%%%%%%%%%%%%%%%%%%%%%%%%%%%%%%%%%%%%%%%%%%%%%%%%%%
%%%%%%%%%%%%%%%%%%%%%%%%%%%%%%%%%%%%%%%%%%%%%%%%%%%%%%%%%%%%%%%%%%%%%%%%%%%%%%%%%%%%%%%%%%%%%%%%%%%%%%%%%%%%%%%%%%%%%%%%%%%%%%%%%%%%%%%%%%%%%%%%%%%%%%%%%

\setcounter{equation}{0}
\setcounter{section}{0}
\section{\sc Introduction}
We consider the $\lambda$-periodic Cauchy problem for
\begin{equation}\label{fivp}
\left\{
\begin{aligned}
u_t+u_{xxx}+\eta Lu + uu_x &=0, \qquad x\in [0,\lambda], \;\; t\in [0,+\infty), \\
u(x,0)&=u_0(x),
\end{aligned}
\right.
\end{equation}
where $\eta >0$ is a constant, the linear operator $L$ is defined via the Fourier transform by
\begin{align}
(Lu)^{\wedge}(k)=-\Phi(k)\widehat{u}(k),\qquad \text{where}\quad k\in \mathbb{Z}/\lambda , \label{operadorl}
\end{align}
and the Fourier symbol $\Phi(k)$
%\begin{align}
%\Phi(k)= \sum_{j=0}^{n} \sum_{l=0}^{2m} C_{l,j}\,k^l\,|k|^j; \quad C_{l,j}\in \mathbb{R}, \;\; C_{2m,n}<0 \label{multiplicadordel}
%\end{align}
is a real valued function which is bounded above; i.e., there is a constant $\alpha $ such that $\Phi(k)\leq \alpha $. We take $\alpha \geq 1$ without lost of generality.
\\ \\
%\subsection{\sc Examples}
Before stating the main result of this work we give some important examples that belong to the model considered in (\ref{fivp}), where $u=u(x,t)$ is a real-valued function and $\eta >0$ is a constant.
%\begin{description}
%%%%%%%%%%%%%%%%%%%%%%%%%%%%%%%%%%%%%%%%%%%%%%%%%%%%%%%%%%%%%%%%%%%%%%%%%%%%%%%%%%%%%%%%%%%%%%%%%%%%%%%%%%%%%%%%%%%%%%%%%%%%%%%%%%%%%%%%%%%%%%%%%%%%%%%
%\item[1.]
The first example is the Korteweg-de Vries-Burgers equation
\begin{equation}\label{KdVBur}
\left\{
\begin{aligned}
u_t+u_{xxx}-\eta u_{xx}+uu_x&=0, \;\;\;  t\geq 0,    \\
u(x,0)&=u_0(x).
\end{aligned}
\right.
\end{equation}
Molinet and Ribaud considered the initial value problem (\ref{KdVBur}) in \cite{Mol} and proved that it is globally well-posed for given data in $H^s(\mathbb{R})$, $s>-1$, and ill-posed in $H^s(\mathbb{R})$ for $s<-1$ in the sense that one cannot solve the Cauchy problem for (\ref{KdVBur}) by a Picard iterative method implemented on the integral formulation. They show that these results are also valid in the periodic setting. These results are surprising  because the index $s=-1$ is lower than the exponents $s=-3/4$ and $s=-1/2$ which are boundaries indexes that determine the Sobolev spaces where it is possible to obtain well-posedness results using a Picard iterative method implemented on the integral formulation for the KdV equation on $\mathbb{R}$ and $\mathbb{T}$, respectively. 
%critical real and critical periodic Sobolev space index $s=-1$ for (\ref{KdVBur}) is lower than the index $s=-3/4$ of the critical real Sobolev space for the KdV equation and also lower than the index $s=-1/2$ of the critical periodic Sobolev space for the KdV equation. 
This was the first almost sharp result to a dispersive-dissipative equation using the Fourier restriction norm method or Bourgain method. It is not known what happen when $s=-1$.
%The equation (\ref{KdVBur}) is also known as the parabolic regularization of the KdV equation with $\eta >0$. Some years ago, when the interest was to obtain local results for given data in larger Sobolev spaces, this regularization was used to obtain well-posedness results for $\eta >0$ and then pass the limit $\eta \downarrow 0$. However, this limit is a delicate matter.
%%%%%%%%%%%%%%%%%%%%%%%%%%%%%%%%%%%%%%%%%%%%%%%%%%%%%%%%%%%%%%%%%%%%%%%%%%%%%%%%%%%%%%%%%%%%%%%%%%%%%%%%%%%%%%%%%%%%%%%%%%%%%%%%%%%%%%%%%%%%%%%%%%%%%%%
%\item[2.]
\\ \\
Other model that fits in the family (\ref{fivp}) is the derivative Korteweg-de Vries-Kuramoto Sivashinsky equation
\begin{equation}\label{KdVKS}
\left\{
\begin{aligned}
u_t+u_{xxx} + \eta (u_{xx}+u_{xxxx}) + uu_x &=0, \;\;\;\;t\geq 0, \\
u(x,0)&=u_0(x),
\end{aligned}
\right.
\end{equation}
This equation arises as a model for long waves in a viscous fluid flowing down an inclined plane and also describes drift waves in a plasma (cf. \cite{Cohen, Topper}). The equation (\ref{KdVKS}) is a particular case of Benney-Lin equation \cite{Benney, Topper}; i.e.,
\begin{equation}\label{blin}
\left\{
\begin{aligned}
u_t+u_{xxx}+\eta (u_{xx} + u_{xxxx})+ \beta u_{xxxxx} + uu_x&=0, \;\;\;x\in \mathbb{R}, \;\;\; t\geq 0,   \\
u(x,0)&=u_0(x),
\end{aligned}
\right.
\end{equation}
when $\beta=0$. The initial value problem associated to (\ref{KdVKS}) was studied by Biagioni, Bona, Iorio and Scialom in \cite{BBIS}. They also determined the limiting behavior of solutions as the dissipation tends to zero. Biagioni and Linares proved global well-posedness for the initial value problem (\ref{blin}) for initial data in $L^2(\mathbb{R})$ in \cite{bl}. The Benney-Lin equation was studied by Chen and Li in \cite{chenli} using the Fourier restriction norm method too. They proved that (\ref{blin}) is globally well posed in the Sobolev spaces $H^s(\mathbb{R})$ for $0\geq s >-2$ and ill-posed in $H^s(\mathbb{R})$ for $s<-2$ in the sense that one cannot solve the Cauchy problem for (\ref{blin}) by a Picard iterative method implemented on the integral formulation.
%%%%%%%%%%%%%%%%%%%%%%%%%%%%%%%%%%%%%%%%%%%%%%%%%%%%%%%%%%%%%%%%%%%%%%%%%%%%%%%%%%%%%%%%%%%%%%%%%%%%%%%%%%%%%%%%%%%%%%%%%%%%%%%%%%%%%%%%%%%%%%%%%%%%%%%
%\item[3.]
\\ \\
Another example of this type is the Ostrovsky-Stepanyams-Tsimring (OST) equation:
\begin{equation}\label{ost}
\left\{
\begin{aligned}
u_t+u_{xxx}-\eta (\mathcal{H}u_x + \mathcal{H}u_{xxx}) + u^pu_x&=0, \qquad t\geq 0, \;\;\; p=1  \\
u(x,0)&=u_0(x),
\end{aligned}
\right.
\end{equation}
where $\mathcal{H}$ denotes the Hilbert transform:
\begin{equation}
\mathcal{H} f(x)=-\dfrac{1}{\pi}v.p.\dfrac{1}{x}\ast f=\dfrac{1}{\pi}\lim_{\epsilon\to 0}\int_{|x|>\epsilon}\dfrac{f(y)}{x-y}\,dy \, \label{Hilbert}
\end{equation}
The equation (\ref{ost}), with $p=1$, was derived by Ostrovsky-et al. in \cite{OST} to describe the radiational instability of long waves in a stratified shear flow. The earlier well-posedness results for (\ref{ost}), with $p=1$, can be found in \cite{Borys}, for given data in $H^s(\mathbb{R})$, local result when $s>1/2$ and global result for $s\geq 1$. Carvajal and Scialom in \cite{CS} considered the initial value problem (\ref{ost}) in the real case and proved the local well-posedness results for given data in $H^s(\mathbb{R})$, $s\geq 0$ when $p=1,2,3$. They also obtained the global well-posedness results for data in $L^2(\mathbb{R})$ with $p=1$. In \cite{cuiz, cuizhao} Cui and Zhao obtained a low regularity result on the (\ref{ost}) with $p=1$ by Fourier restriction norm method. Indeed, they proved that the initial value problem (\ref{ost}) is locally well-posed in $H^s(\mathbb{R})$ for $s>-1$. Finnally, Zhao in \cite{zhao} proved that (\ref{ost}) is locally well-posed in $H^s(\mathbb{R})$ for $s>-5/4$.
%%%%%%%%%%%%%%%%%%%%%%%%%%%%%%%%%%%%%%%%%%%%%%%%%%%%%%%%%%%%%%%%%%%%%%%%%%%%%%%%%%%%%%%%%%%%%%%%%%%%%%%%%%%%%%%%%%%%%%%%%%%%%%%%%%%%%%%%%%%%%%%%%%%%%%%
%\item[4.]
\\ \\
The next Cauchy problem of a dissipative version of the KdV equation with rough initial data
\begin{equation}\label{kdvdissipativa}
\left\{
\begin{aligned}
u_t+u_{xxx}+Lu + uu_x&=0, \qquad t\geq 0,   \\
u(x,0)&=u_0(x),
\end{aligned}
\right.
\end{equation}
where $L=|\partial_x|^{2\gamma}$ is defined by a multiplier with symbol $|k|^{2\gamma}$ and $\gamma\geq 1$, is other example that belongs to the class (\ref{fivp}). (\ref{kdvdissipativa}) was studied by Han and Peng in \cite{hanpeng}. They proved working in Bourgain type space the local and global well posedness results for Sobolev spaces $H^s(\mathbb{R})$ of negative order, and the order number is lower than the well known value $-\frac{3}{4}$, i.e., $s>-s_{\gamma}$, where $s_{\gamma}$ denotes the boundary index and it is given by:
\begin{equation*}
\left\{
\begin{aligned}
\dfrac{3-\gamma}{4-2\gamma},\qquad &\text{if}\qquad 1\leq \gamma \leq \dfrac{3}{2},   \\
\gamma,\qquad &\text{if}\qquad \gamma > \dfrac{3}{2}.
\end{aligned}
\right.
\end{equation*}
When $\gamma=1$, this result agrees with that in \cite{Mol}, and it improves the result obtained in \cite{moliriba} in the case $\gamma\geq 1$.
%%%%%%%%%%%%%%%%%%%%%%%%%%%%%%%%%%%%%%%%%%%%%%%%%%%%%%%%%%%%%%%%%%%%%%%%%%%%%%%%%%%%%%%%%%%%%%%%%%%%%%%%%%%%%%%%%%%%%%%%%%%%%%%%%%%%%%%%%%%%%%%%%%%%%%%
\\ \\
Carvajal and Panthee proved  in \cite{CP} local well-posedness in Sobolev spaces $H^s(\mathbb{R})$ with $s>-3/4$ to the initial value problem (\ref{fivp}) but only to the real case. In particular, they obtained that result when the symbol $\Phi$ is given by
\begin{equation*}
\Phi(\xi)= \sum_{j=0}^{n} \sum_{l=0}^{2m} C_{l,j}\,\xi^l\,|\xi|^j; \quad C_{l,j}\in \mathbb{R}, \;\; C_{2m,n}=-1 .
\end{equation*}
The examples above correspond to this case. They followed the theory developed by Bourgain \cite{Bourgain} and Kenig, Ponce and Vega \cite{KPV}. They used the usual Bourgain's space associated to the KdV equation instead of the Bourgain's space associated to the linear part of the initial value problem (\ref{problema}).
%%%%%%%%%%%%%%%%%%%%%%%%%%%%%%%%%%%%%%%%%%%%%%%%%%%%%%%%%%%%%%%%%%%%%%%%%%%%%%%%%%%%%%%%%%%%%%%%%%%%%%%%%%%%%%%%%%%%%%%%%%%%%%%%%%%%%%%%%%%%%%%%%%%%%%%
%%%%%%%%%%%%%%%%%%%%%%%%%%%%%%%%%%%%%%%%%%%%%%%%%%%%%%%%%%%%%%%%%%%%%%%%%%%%%%%%%%%%%%%%%%%%%%%%%%%%%%%%%%%%%%%%%%%%%%%%%%%%%%%%%%%%%%%%%%%%%%%%%%%%%%%
%%%%%%%%%%%%%%%%%%%%%%%%%%%%%%%%%%%%%%%%%%%%%%%%%%%%%%%%%%%%%%%%%%%%%%%%%%%%%%%%%%%%%%%%%%%%%%%%%%%%%%%%%%%%%%%%%%%%%%%%%%%%%%%%%%%%%%%%%%%%%%%%%%%%%%%
%%%%%%%%%%%%%%%%%%%%%%%%%%%%%%%%%%%%%%%%%%%%%%%%%%%%%%%%%%%%%%%%%%%%%%%%%%%%%%%%%%%%%%%%%%%%%%%%%%%%%%%%%%%%%%%%%%%%%%%%%%%%%%%%%%%%%%%%%%%%%%%%%%%%%%%

\subsection{\textsc{Notation and Main Result}}
We recall the theory developed by T. Tao in \cite{tao}. We define the Fourier transform of a function $f$ defined on $[0,\lambda]$ by
\begin{align}
\widehat{f}(k)&=\int_0^{\lambda}e^{-2\pi ikx}\,f(x)\,dx \label{fouriertransfper}\\
\intertext{and we have the Fourier inversion formula}
f(x)&=\int e^{2\pi ikx}\,\widehat{f}(k)\,(dk)_{\lambda} \label{fourierinvform}\\
\intertext{where $(dk)_{\lambda}$ is the normalized counting measure on $\mathbb{Z}/\lambda$ given by}
\int a(k)\,(dk)_{\lambda}&=\frac{1}{\lambda}\,\sum_{k\in \mathbb{Z}/\lambda} a(k) .\label{countmeasure} \\
\intertext{The usual properties of the Fourier transform hold:}
\nor{f}{L^2([0,\lambda])}&=\nor{\widehat{f}}{L^2((dk)_{\lambda})} \qquad \text{(Plancherel),} \label{plancherel} \\
\int_0^{\lambda}f(x)\,\bar{g(x)}\,dx &= \int \widehat{f}(k)\,\bar{\widehat{g}}(k)\,(dk)_{\lambda} \qquad \text{(Parseval),} \label{parseval} \\
\widehat{fg}(k)=\widehat{f}\ast_{\lambda}\widehat{g}\,(k)&=\int \widehat{f}(k-k_1)\,g(k_1)\,(dk_1)_{\lambda} \qquad \text{(Convolution),} \label{convolution} \\
\intertext{and so on. If we apply $\partial_x^m$, $m\in \mathbb{N}$, to (\ref{fourierinvform}), we obtain}
\partial_x^mf(x)&=\int e^{2\pi ikx}\,(2\pi ik)^m\,\widehat{f}(k)\,(dk)_{\lambda}. \label{transfderiv} \\
\intertext{This, together with (\ref{plancherel}), motivates us to define the Sobolev space $H^s([0,\lambda])$ with the norm}
\nor{f}{H^s([0,\lambda])}&=\nor{\langle k\rangle^s\,\widehat{f}(k)}{L^2((dk)_{\lambda})}. \label{sobolevnormper}
\end{align}
We will often denote this space by $H^s_{\lambda}$ for simplicity. For a function $v=v(x,t)$ which is $\lambda$-periodic with respect to the $x$ variable and with the time variable $t\in \mathbb{R}$, we define the space-time Fourier transform $\widehat{v}=\widehat{v}(k,\tau)$ for $k\in \mathbb{Z}/\lambda$ and $\tau \in \mathbb{R}$ by
\begin{align}
\widehat{v}(k,\tau)&=\int_{\mathbb{R}} \int_0^{\lambda} e^{-2\pi ikx}\,e^{-2\pi i\tau t}\,v(x,t)\,dx\,dt. \label{fouriertransfxt} \\
\intertext{This transform is inverted by}
v(x,t)&=\int_{\mathbb{R}} \int e^{2\pi ikx}\,e^{2\pi i\tau t}\,\widehat{v}(k,\tau)\,(dk)_{\lambda}\,d\tau. \label{fourierinvxt}
\end{align}
Similarly, $\widehat{v}(k,t)$ and $\widehat{v}(x,\tau)$ will denote the Fourier transform of $v(x,t)$ respect to the variables $x$ and $t$, respectively. $C$ will be denote a positive constant which may be different even in a single chain of inequalities. If $X, Y$ are Banach spaces, $\mathcal{B}(X;Y)$ is the space of the linear continue operators of $X$ in $Y$ with the norm  $\left\|T\right\|_{X\to Y}=\sup_{\left\|x\right\|_X=1}\left\|Tx\right\|_Y$. If $X=Y$ we will write $\mathcal{B}(X)$ inside of $\mathcal{B}(X;X)$.
The solution to the linear KdV equation:
\begin{equation}\label{kdvlinear}
\left\{
\begin{aligned}
u_t + u_{xxx}&=0, \quad x\in [0,\lambda], \quad t\in \mathbb{R}, \\
u(x,0)&=u_0(x),
\end{aligned}
\right.
\end{equation}
is given by
\begin{align}
u(x,t)=U_{\lambda}(t)\,u_0(x)&=\int e^{2\pi ikx}\,e^{-(2\pi ik)^3t}\,\widehat{u_0}(k)\,(dk)_{\lambda}, \label{linearkdvsol} \\
\intertext{which may be rewritten as a space-time inverse Fourier transform,}
U_{\lambda}(t)\,u_0(x)&=\int_{\mathbb{R}}\int e^{2\pi i\tau t}\,e^{2\pi ikx}\,\delta(\tau -4\pi^2 k^3)\,\widehat{u_0}(k)\,(dk)_{\lambda}\,d\tau, \label{grupokdv}
\end{align}
where $\delta(\kappa)$ represents a 1-dimensional Dirac mass at $\kappa=0$. This shows that $U_{\lambda}(\cdot)\,u_0$ has its space-time Fourier transform supported precisely on the cubic $\tau=4\pi^2k^3$ in $\mathbb{Z}/\lambda \times \mathbb{R}$. So, we recall the known Bourgain's space associated to the KdV equation. For $s$, $b \in \mathbb{R}$, we define the $\mathcal{Y}_{s,b}([0,\lambda]\times \mathbb{R})$ spaces for $\lambda$-periodic KdV via the norm
\begin{equation}\label{normakdv}
\begin{split}
\nor{u}{\mathcal{Y}_{s,b}([0,\lambda]\times \mathbb{R})}&\equiv \nor{\langle\tau - 4\pi^2k^3\rangle^b\,\langle k\rangle^s\,\widehat{u}(k,\tau)}{L^2((dk)_{\lambda})L_{\tau}^2}=\nor{\langle\tau\rangle^b\,\langle k\rangle^s\,(U_{\lambda}(-t)u)^{\wedge}(k,\tau)}{L^2((dk)_{\lambda})L_{\tau}^2} \\
&=\biggl(\int\int_{\mathbb{R}}\langle\tau \rangle^{2b}\,\langle k\rangle^{2s}\,|(U_{\lambda}(-t)u)^{\wedge}(k,\tau)|^2\,d\tau \,(dk)_{\lambda} \biggr)^{1/2}.
\end{split}
\end{equation}
\begin{rem}\label{meanzero}
The spatial mean $\int_{\mathbb{T}}u(x,t)\,dx$ is conserved during the evolution of the KdV equation. We may assume that the initial data $\phi$ satisfies a mean-zero assumption $\int_{\mathbb{T}}\phi(x)\,dx$ since otherwise we can replace the dependent variable $u$ by $v=u-\int_{\mathbb{T}}\phi$ at the expense of a harmless linear first order term. This observation was used by Bourgain in \cite{Bourgain}.
%The mean-zero assumption is crucial for some of the analysis that follows.
\end{rem}
Since the $\lambda$-periodic initial value problem for KdV is equivalent to the integral equation
\begin{equation}\label{inteqkdv}
u(t)=U_{\lambda}(t)\phi -\frac{1}{2}\int_0^tU_{\lambda}(t-t')\,\partial_x(u^2(t'))\,dt',
\end{equation}
the study of periodic KdV in \cite{Bourgain} and \cite{KPV} was based in solve (\ref{inteqkdv}) using the contraction principle in the Bourgain's spaces $\mathcal{Y}_{s,1/2}$ which was possible in virtue of the optimal bilinear estimate for $\partial_xu^2$, from Kenig, Ponce and Vega in the periodic case:
\begin{prop}[\cite{KPV}] \label{bilest}
For $s \in (-1/2,0]$  it follows that
$$\nor{\frac{1}{2}\partial_xu^2}{\mathcal{Y}_{s,-1/2}}\leq c \nora{u}{\mathcal{Y}_{s,1/2}}{2}.$$
\end{prop}
For the case $s=-1/2$, see the Corollary 6.5 in \cite{tao}. So, it was proved that the initial value problem for KdV on $\mathbb{T}$ is locally well-posed for $s\geq -1/2$. The space $\mathcal{Y}_{s,1/2}$ barely fails to control the $L_t^{\infty}H_x^s$ norm. To ensure continuity of the time flow of the solution Colliander, Keel, Staffilani, Takaoka and Tao, in \cite{CKSTT}, introduced the slightly smaller space $Y^s$ defined via the norm
\begin{align}
\nor{u}{Y^s}&=\nor{u}{\mathcal{Y}_{s,1/2}}+\nor{\langle k\rangle^s\,\widehat{u}(k,\tau)}{L^2((dk)_{\lambda})L^1(d\tau)}, \label{normys} \\
\intertext{and the companion space $Z^s$ defined via the norm}
\nor{u}{Z^s}&=\nor{u}{\mathcal{Y}_{s,-1/2}}+\nor{\frac{\langle k\rangle^s\,\widehat{u}(k,\tau)}{\langle\tau -4\pi^2k^3\rangle}}{L^2((dk)_{\lambda})L^1(d\tau)}. \label{normzs}
\end{align}
Note that, if $u\in Y^s$, then $u\in L_{t}^{\infty}H_x^s$. Thus, they solve the integral equation (\ref{inteqkdv}) based around iteration in the space $Y^s$. They obtained the bilinear estimate for $\partial_xu^2$:
\begin{prop}
If $u$ and $v$ are $\lambda$-periodic functions of $x$, also depending upon $t$ having zero $x$-mean for all $t$, then
\begin{equation}\label{bilestckstt}
\nor{\Psi(t)\,\partial_x(uv)}{Z^{-1/2}}\lesssim \lambda^{0+}\,\nor{u}{\mathcal{Y}_{-1/2,1/2}}\,\nor{v}{\mathcal{Y}_{-1/2,1/2}}.
\end{equation}
where $\Psi \in C_0^{\infty}(\mathbb{R})$ is a cut-off function such that $0\leq \Psi(t)\leq 1$ and is supported on $[-2,2]$ with $\Psi=1$ on $[-1,1]$.
\end{prop}
\begin{rem}
Note that (\ref{bilestckstt}) implies $\nor{\Psi(t)\,\partial_x(uv)}{Z^{-1/2}}\lesssim \lambda^{0+}\,\nor{u}{Y^{-1/2}}\,\nor{v}{Y^{-1/2}}.$
\end{rem}
So, Colliander, Keel, Staffilani, Takaoka and Tao in \cite{CKSTT} reproved that the initial value problem for KdV equation on $\mathbb{T}$ is locally well-posed for $s\geq -1/2$. Our interest here is to obtain well-posedness results for the $\lambda$-periodic initial value problem (\ref{fivp}) with given data $u_0$ in the Sobolev space $H^s_{\lambda}$ of negative order:
\begin{teor}[Main Result]\label{maintheorem}
The initial value problem (\ref{fivp}) with $\eta>0$ and $L$ given by (\ref{operadorl}) is locally well-posed for any data $u_0\in H^s(\mathbb{T})$, for $s\geq -1/2$.
\end{teor}
To prove this theorem we use Bourgain's type space. So, we should be able to write (\ref{fivp}) for all $t\in \mathbb{R}$. For this, we define
\begin{equation}\label{etadet}
\eta (t) \equiv \eta \, sgn(t) = \left\{
\begin{aligned}
\eta&, \quad \text{if}\;\; t\geq 0, \\
-\eta&, \quad \text{if}\;\; t<0,
\end{aligned}
\right.
\end{equation}
and write (\ref{fivp}) in the form
\begin{equation}\label{problema}
\left\{
\begin{aligned}
u_t+u_{xxx}+\eta(t) Lu + uu_x &=0, \qquad x\in [0,\lambda], \;\;\; t\in \mathbb{R}, \\
u(x,0)&=u_0(x).
\end{aligned}
\right.
\end{equation}
We first want to build a representation formula for the solution of the linearization of (\ref{fivp}) about the zero solution. So, we wish to solve the linear homogeneous $\lambda$-periodic initial value problem
\begin{equation}\label{fivplinear}
\left\{
\begin{aligned}
w_t+w_{xxx}+\eta(t) Lw &=0, \qquad x\in [0,\lambda], \;\;\; t\in \mathbb{R}, \\
w(x,0)&=w_0(x).
\end{aligned}
\right.
\end{equation}
The Fourier inversion formula (\ref{fourierinvform}) allows us to write the solution of (\ref{fivplinear}):
\begin{equation}
w(x,t)=V_{\lambda}(t)\,w_0(x)=\int e^{2\pi ikx}\,e^{-(2\pi ik)^3t+\eta \Phi(k)|t|}\,\widehat{w_0}(k)\,(dk)_{\lambda}. \label{linearsol}
\end{equation}
%where the semigroup $V_{\eta}(t)$ is defined as
%\begin{equation}
%\widehat{V_{\eta}(t)\,\phi}(k)=e^{itk^3+\eta |t| \Phi(k)}\widehat{u_0}(k), \quad k\in \mathbb{Z}. \label{veta}
%\end{equation}
Observe that, defining $\widetilde{U_{\lambda}}(t)$ by
$$(\widetilde{U_{\lambda}}(t)w_0)^{\wedge}(k)=e^{\eta |t| \Phi(k)}\widehat{w_0}(k), \quad k\in \mathbb{Z}/\lambda,$$
the semigroup $V_{\lambda}(t)$ can be written as $V_{\lambda}(t)=U_{\lambda}(t)\,\widetilde{U_{\lambda}}(t)$ where $U_{\lambda}(t)$ is the unitary group of the KdV (\ref{grupokdv}).
We next find a representation for the solution of the linear inhomogeneous $\lambda$-periodic initial value problem
\begin{equation}\label{fivpinhom}
\left\{
\begin{aligned}
v_t+v_{xxx}+\eta(t) Lv &=F, \qquad x\in [0,\lambda], \;\;\; t\in \mathbb{R}, \\
v(x,0)&=0,
\end{aligned}
\right.
\end{equation}
with $F=F(x,t)$ a given time-dependent $\lambda$-periodic (in $x$) function. By Duhamel's principle,
\begin{equation}
v(x,t)=\int_0^tV_{\lambda}(t-t')\,F(x,t')\,dt'. \label{inhomsol}
\end{equation}
We apply (\ref{linearsol}), rewrite $\widehat{F}(k,t')$ using the Fourier inversion formula in the time variable and rearrange integrations to find
\begin{equation}
v(x,t)=\int_{\mathbb{R}}\int e^{2\pi ikx}\,e^{2\pi i(4\pi^2k^3)t+\eta \Phi(k)|t|}\,\int_0^te^{[2\pi i(\tau - 4\pi^2k^3)-\eta(t)\Phi(k)]t'}\,dt'\;\widehat{F}(k,\tau)\,(dk)_{\lambda}\,d\tau .\label{inhomsoluno}
\end{equation}
Performing the $t'$-integration, we find
\begin{equation}
v(x,t)=\int_{\mathbb{R}}\int e^{2\pi ikx}\,e^{2\pi i(4\pi^2k^3)t+\eta \Phi(k)|t|}\,\frac{e^{[2\pi i(\tau - 4\pi^2k^3)-\eta(t)\Phi(k)]t}-1}{2\pi i(\tau - 4\pi^2k^3)-\eta(t)\Phi(k)}\;\widehat{F}(k,\tau)\,(dk)_{\lambda}\,d\tau .\label{inhomsoldos}
\end{equation}
Then, the $\lambda$-periodic initial value problem for (\ref{problema}) is equivalent to the integral equation
\begin{equation}
u(t)=V_{\lambda}(t)u_0-\frac{1}{2}\int_0^tV_{\lambda}(t-t')\,\partial_x(u^2(t'))\,dt'. \label{integralequation}
\end{equation}
The integral equation (\ref{integralequation}) can be solved using the contraction principle in the space $\mathcal{Y}_{s,1/2}$ following the ideas of Carvajal and Panthee in \cite{CP}. The main difficulty to resolve it of this way is the periodic bilinear estimate for $\partial_xu^2$, given in the Proposition \ref{bilest}, because it's very restrictive compared with the bilinear estimate (see Theorem 1.1 in \cite{KPV}) of the real case in which $b\in (1/2,1)$. So, we shall obtain a refined estimative to the forcing term of the integral equation associated to (\ref{problema}) which will permits us to use the Proposition \ref{bilest} to solve (\ref{integralequation}). This refinement is made in the Proposition \ref{forcingterm} but we prove our main result via the contraction principle in the space $Y^s$ and the bilinear estimate (\ref{bilestckstt}) from Colliander, Keel, Staffilani, Takaoka and Tao.
\\ \\
The layout of this paper is as follows. In Section $2$ we present some basic results. In Section $3$ we give the boundedness results for linear operators involving the spaces $Y^s$, $Z^s$ and $\mathcal{Y}_{s,1/2}$. The proof of the main Theorem \ref{maintheorem} will be given in Section $4$.

%%%%%%%%%%%%%%%%%%%%%%%%%%%%%%%%%%%%%%%%%%%%%%%%%%%%%%%%%%%%%%%%%%%%%%%%%%%%%%%%%%%%%%%%%%%%%%%%%%%%%%%%%%%%%%%%%%%%%%%%%%%%%%%%%%%%%%%%%%%%%%%%%%%%%%%%%
%%%%%%%%%%%%%%%%%%%%%%%%%%%%%%%%%%%%%%%%%%%%%%%%%%%%%%%%%%%%%%%%%%%%%%%%%%%%%%%%%%%%%%%%%%%%%%%%%%%%%%%%%%%%%%%%%%%%%%%%%%%%%%%%%%%%%%%%%%%%%%%%%%%%%%%%%
%%%%%%%%%%%%%%%%%%%%%%%%%%%%%%%%%%%%%%%%%%%%%%%%%%%%%%%%%%%%%%%%%%%%%%%%%%%%%%%%%%%%%%%%%%%%%%%%%%%%%%%%%%%%%%%%%%%%%%%%%%%%%%%%%%%%%%%%%%%%%%%%%%%%%%%%%
%%%%%%%%%%%%%%%%%%%%%%%%%%%%%%%%%%%%%%%%%%%%%%%%%%%%%%%%%%%%%%%%%%%%%%%%%%%%%%%%%%%%%%%%%%%%%%%%%%%%%%%%%%%%%%%%%%%%%%%%%%%%%%%%%%%%%%%%%%%%%%%%%%%%%%%%%
%%%%%%%%%%%%%%%%%%%%%%%%%%%%%%%%%%%%%%%%%%%%%%%%%%%%%%%%%%%%%%%%%%%%%%%%%%%%%%%%%%%%%%%%%%%%%%%%%%%%%%%%%%%%%%%%%%%%%%%%%%%%%%%%%%%%%%%%%%%%%%%%%%%%%%%%%

\setcounter{equation}{0}
\section{\textsc{Preliminary Results}}

\begin{lema}\label{lema}
Let $a \le 0$, $\psi\in C_0^{\infty}$ with support in $[-2,2]$ and $\psi_T(t)=\psi(t/T)$. Then,
\begin{align}
\nor{\psi_T(t)\,\int_0^te^{a|t-x|}g(x)\,dx}{L^2}&\leq \dfrac{C\,(1+T)}{1+|a|}\,\nor{g}{L^2}. \label{iresult} \\
\intertext{If $g(0)=0$,}
\nor{\psi_T(t)\,\frac{d}{dt}\;\int_0^te^{a|t-x|}\,g(x)\,dx}{L^2}&\leq \dfrac{C\,(1+T)}{1+|a|}\;\nor{\frac{dg}{dt}}{L^2}, \label{iiresult} \\
\intertext{and,}
\nor{sgn(\cdot)\,g(\cdot)}{H^1}&\leq \nor{g(\cdot)}{H^1}. \label{iiiresult}
\end{align}
$C=C_{\psi}=\max \Bigl\{\nor{\psi}{L^{\infty}}, \nor{\frac{d\psi}{dt}}{L^{\infty}}\Bigr\}$ is a constant depending on $\psi$.
\end{lema}
\begin{proof}[Proof]
We are going to argue by duality to obtain (\ref{iresult}). We take $\varphi\in L^2$ with $\nor{\varphi}{L^2}\leq 1$. Then,
\begin{align}
\int_{\mathbb{R}}\varphi(t)\,\Bigl\{\psi_T(t)&\int_0^te^{a|t-x|}g(x)\,dx\Bigr\}\,dt= \int_{-2T}^{2T} \int_0^t\varphi(t)\,\{\psi_T(t)e^{a|x|}g(t-x)\,dx\}\,dt \notag \\
&=\int_{-2T}^0e^{a|x|}\int_{-2T}^x\varphi(t)\psi_T(t)g(t-x)\,dt\,dx + \int_0^{2T}e^{a|x|}\int_x^{2T}\varphi(t)\psi_T(t)g(t-x)\,dt\,dx \notag \\
&\leq 2\,\int_{-2T}^{2T}e^{a|x|}\,\nor{\varphi}{L^2}\,\nor{\psi_T(\cdot)g(\cdot -x)}{L^2}\,dx\leq \frac{C}{|a|}\,\nor{g}{L^2}. \label{causch}
\end{align}
It was used the Cauchy-Schwartz's inequality to obtain the first inequality in (\ref{causch}). Also, it is truth that 
\begin{align}    
\int_{-2T}^{2T}e^{a|x|}\,\nor{\varphi}{L^2}\,\nor{\psi_T(\cdot)g(\cdot -x)}{L^2}\,dx &\leq C\,T\,\nor{g}{L^2}. \label{causchuno}
\end{align}
From (\ref{causch}) and (\ref{causchuno}), we conclude (\ref{iresult}). We use that $g(0)=0$ to obtain
\begin{align*}
\frac{d}{dt}\biggl(\int_0^te^{a|t-x|}\,g(x)\,dx\biggr)= \int_0^te^{a|x|}&\,\frac{dg}{dt}(t-x)\,dx= \int_0^te^{a|t-x|}\,\frac{dg}{dx}(x)\,dx
\end{align*}
and this, together with (\ref{iresult}), implies (\ref{iiresult}). An easy computation shows (\ref{iiiresult}).
%We can write $sgn(\cdot)\,g(\cdot)=\chi_{\mathbb{R}^+}(\cdot)\,g(\cdot)-\chi_{\mathbb{R}^-}(\cdot)\,g(\cdot)$ so, we use the argument in \cite{Mol} on the page 1989.
\end{proof}
%%%%%%%%%%%%%%%%%%%%%%%%%%%%%%%%%%%%%%%%%%%%%%%%%%%%%%%%%%%%%%%%%%%%%%%%%%%%%%%%%%%%%%%%%%%%%%%%%%%%%%%%%%%%%%%%%%%%%%%%%%%%%%%%%%%%%%%%%%%%%%%%%%%%%%%%
\begin{rem}\label{notaciondepsi}
We consider a cut-off function $\Psi \in C^{\infty}(\mathbb{R})$, such that $0\leq \Psi(t)\leq 1$,
\[\Psi(t)=\begin{cases}
1, \;\;\;& \text{if} \;\;|t|\leq 1   \\
0, \;\;\;& \text{if} \;\;|t|\geq 2.  \label{Psi}
\end{cases}\]
Let us define $\Psi_{T}(t)=\Psi(\frac{t}{T})$ and $\widetilde{\Psi_{T}}(t)= sgn(t)\Psi_{T}(t)$. Note that multiplication by $\Psi(t)$ is a bounded operation on the spaces $Y^s$, $Z^s$ and $\mathcal{Y}_{s,b}$.
\end{rem}
%%%%%%%%%%%%%%%%%%%%%%%%%%%%%%%%%%%%%%%%%%%%%%%%%%%%%%%%%%%%%%%%%%%%%%%%%%%%%%%%%%%%%%%%%%%%%%%%%%%%%%%%%%%%%%%%%%%%%%%%%%%%%%%%%%%%%%%%%%%%%%%%%%%%%%%%%%
%%%%%%%%%%%%%%%%%%%%%%%%%%%%%%%%%%%%%%%%%%%%%%%%%%%%%%%%%%%%%%%%%%%%%%%%%%%%%%%%%%%%%%%%%%%%%%%%%%%%%%%%%%%%%%%%%%%%%%%%%%%%%%%%%%%%%%%%%%%%%%%%%%%%%%%%%%
%%%%%%%%%%%%%%%%%%%%%%%%%%%%%%%%%%%%%%%%%%%%%%%%%%%%%%%%%%%%%%%%%%%%%%%%%%%%%%%%%%%%%%%%%%%%%%%%%%%%%%%%%%%%%%%%%%%%%%%%%%%%%%%%%%%%%%%%%%%%%%%%%%%%%%%%%%
%%%%%%%%%%%%%%%%%%%%%%%%%%%%%%%%%%%%%%%%%%%%%%%%%%%%%%%%%%%%%%%%%%%%%%%%%%%%%%%%%%%%%%%%%%%%%%%%%%%%%%%%%%%%%%%%%%%%%%%%%%%%%%%%%%%%%%%%%%%%%%%%%%%%%%%%%%
The next result will allow us to prove the Lemma \ref{chave6} and to reduce the proof of (\ref{sothankgodone}) .
\begin{prop}\label{fermin}
Let $0\leq b\leq 1$, $\alpha_1$, $\alpha_2$ negatives and $a=\alpha_1+\alpha_2$. Then,
\begin{equation}
\nor{\Psi_T(t)\,\int_0^te^{a|t-x|}\,f(x)\,dx}{H^b}\leq C\,(1+T)\,\nor{\Psi_{2T}(t)\int_0^te^{\alpha_2\,|t-x|}\,f(x)\,dx}{H^b} ,\label{chave7}
\end{equation}
where $C=C_{\Psi}=\max \Bigl\{\nor{\Psi}{L^{\infty}}, \nor{\frac{d\Psi}{dt}}{L^{\infty}}\Bigr\}$ is a constant depending on $\Psi$.
\end{prop}
\begin{proof}
Let $g(t)=\int_0^te^{\alpha_2|t-x|}\,f(x)\,dx$. Thus, $\frac{dg}{dt}(t)=f(t)+ \alpha_2\,sgn(t)\,g(t)$. Integrating by parts, we have
\begin{align}
\int_0^te^{a|t-x|}\,f(x)\,dx &= \int_0^te^{a(|t|-|x|)}\,\frac{dg}{dx}(x)\,dx - \alpha_2\,\int_0^te^{a|t-x|}\,sgn(x)\,g(x)\,dx \notag \\
&= g(t)+\alpha_1\,sgn(t)\,\int_0^te^{a|t-x|}\,g(x)\,dx .\label{arie}
\end{align}
We obtain (\ref{chave7}) when $b=0$ as consequence of (\ref{arie}), (\ref{iresult}) and
\begin{equation*}
\nor{\alpha_1\,sgn(t)\,\Psi_T(t)\,\int_0^te^{a|t-x|}\,g(x)\,dx}{L^2}\leq \dfrac{|\alpha_1|\,C\,(1+T)}{1+|\alpha_1|+|\alpha_2|}\nor{\Psi_{2T}\,g}{L^2}.
\end{equation*}
Now, we are going to obtain (\ref{chave7}) when $b=1$. We know from (\ref{arie}) that
\begin{equation}\label{sandra}
\nor{\Psi_T(t)\,\int_0^te^{a|t-x|}\,f(x)\,dx}{H^1}\leq \nor{\Psi_T(t)\,g(t)}{H^1}+|\alpha_1|\nor{sgn(t)\,\Psi_T(t)\,\int_0^te^{a|t-x|}\,g(x)\,dx}{H^1}.
\end{equation}
Since $\nor{\Psi_T\,g}{H^1}\leq C\,\nor{\Psi_{2T}\,g}{H^1}$, by virtue of (\ref{iiiresult}) it is sufficient to estimate $$\nor{\frac{d}{dt}\biggl(\Psi_{T}(t)\,\int_0^te^{a|t-x|}\,g(x)\,dx\biggr)}{L^2}$$
which is bounded by
\begin{equation}\label{pilas}
\nor{\frac{d\Psi_T}{dt}(t)\,\biggl(\int_0^te^{a|t-x|}\,g(x)\,dx\biggr)}{L^2}+\nor{\Psi_T(t)\,\frac{d}{dt}\biggl(\int_0^te^{a|t-x|}\,g(x)\,dx\biggr)}{L^2}.
\end{equation}
For the first term above we can apply (\ref{iresult}) and
\begin{align*}
\nor{\frac{d\Psi}{dt}(t/T)\;\int_0^Te^{a|t-x|}\,(\Psi_{2T}\,g)(x)\,dx}{L^2} \leq \dfrac{C(T+1)}{1+|a|}\,\nor{\Psi_{2T}\,g}{L^2} \leq \dfrac{C\,(T+1)T}{1+|a|}\,\nor{\frac{d}{dt}(\Psi_{2T}\,g)}{L^2},
\end{align*}
where, in the last inequality, it was used that
\begin{equation*}
\nora{\Psi_{2T}\,g}{L^2}{2}=\int_{-4T}^{4T}|\Psi_{2T}(t)\,g(t)|^2\,dt\leq C\,T\,\nora{\Psi_{2T}\,g}{L^{\infty}}{2} \leq C\,T\,\nor{\Psi_{2T}\,g}{L^2}\,\nor{\frac{d}{dt}(\Psi_{2T}\,g)}{L^2}.
\end{equation*}
For the second term from (\ref{pilas}) we used (\ref{iiresult}) with $\Psi_{2T}\,g$ instead $g$ because $g=\Psi_{2T}\,g$ on $[-T, T]$. This implies (\ref{chave7}) when $b=1$. The result (\ref{chave7}) is obtained interpolating the cases $b=0$ and $b=1$.
\end{proof}
%%%%%%%%%%%%%%%%%%%%%%%%%%%%%%%%%%%%%%%%%%%%%%%%%%%%%%%%%%%%%%%%%%%%%%%%%%%%%%%%%%%%%%%%%%%%%%%%%%%%%%%%%%%%%%%%%%%%%%%%%%%%%%%%%%%%%%%%%%%%%%%%%%%%%%%
%%%%%%%%%%%%%%%%%%%%%%%%%%%%%%%%%%%%%%%%%%%%%%%%%%%%%%%%%%%%%%%%%%%%%%%%%%%%%%%%%%%%%%%%%%%%%%%%%%%%%%%%%%%%%%%%%%%%%%%%%%%%%%%%%%%%%%%%%%%%%%%%%%%%%%%%%
%%%%%%%%%%%%%%%%%%%%%%%%%%%%%%%%%%%%%%%%%%%%%%%%%%%%%%%%%%%%%%%%%%%%%%%%%%%%%%%%%%%%%%%%%%%%%%%%%%%%%%%%%%%%%%%%%%%%%%%%%%%%%%%%%%%%%%%%%%%%%%%%%%%%%%%%%
The following Lemma plays a central role estimating the free term of the integral equation (\ref{integralequation}). This Lemma allows us to work in the usual $\mathcal{Y}_{s,1/2}$ space associated to the KdV equation.
\begin{lema}\label{chave}
Let $0< T\lesssim 1$ and $a\leq \alpha$. Then we have
\begin{align}
\nor{\Psi_{T}(\cdot)}{H_t^b}&\leq C (T^{1/2}+T^{1/2-b})\;\;\;\;\; \forall b\geq 0, \label{chave1} \\
\nor{\Psi_{T}(\cdot)e^{a|\cdot|}}{H_t^{1/2}}&\leq C\,e^{2 \alpha}, \label{chave2} \\
\nor{\Psi_{T}(\cdot)e^{a|\cdot|}}{L_t^{1}}&\leq C\,e^{2 \alpha}, \label{novosiete} \\
|(|t|\Psi_T(t)e^{a|t|})^{\wedge}(\tau)|&\leq \dfrac{C\,T^2}{1+(\tau^2+a^2)T^2}, \label{chavenew}
\end{align}
where $C=C_{\psi}=\max \Bigl\{\nor{\psi}{L^{\infty}}, \nor{\frac{d\psi}{dt}}{L^{\infty}}, \nor{\frac{d^2\psi}{dt^2}}{L^{\infty}}\Bigr\}$ is a constant depending on $\psi$.
\end{lema}
\begin{proof}[Proof]
It's clear that
\begin{equation}\label{agodp}
\nora{\Psi_T}{L^2}{2}=\int_{\mathbb{R}}\Bigl|\Psi\Bigl(\frac{t}{T}\Bigr)\Bigr|^2\,dt =\int_{\mathbb{R}}T|\Psi(t)|^2\,dt = T\,\nora{\Psi}{L^2}{2}.
\end{equation}
By the definition of the space $H^b$, we have
\begin{align}
\nor{\Psi_T}{H^b_t}\leq C\,\nor{\Psi_T}{L^2}+C\,\nor{D_t^b\Psi_T}{L^2}=C\,T^{1/2}\nor{\Psi}{L^2}+C\,T^{1/2-b}\nor{D_t^b\Psi}{L^2}, \label{chave3}
\end{align}
where we have used the fact
\begin{align}
\nora{D_t^b\Psi_T}{L^2}{2} &= \int_{\mathbb{R}}|\tau |^{2b}|T\widehat{\Psi}(T\tau)|^2\,d\tau = T^{1-2b}\nora{D_t^b\Psi}{L^2}{2}. \notag \\
\intertext{Since $\nor{\Psi}{L^2}$ and $\nor{D_t^b\Psi}{L^2}$ are bounded by a constant because of the form of the function $\Psi$, then from (\ref{chave3}) we obtain (\ref{chave1}). We call $h(t)=\Psi(t)\,e^{a|t|T}$, and so $h_T(t)=\Psi_T(t)\,e^{a|t|}$, to get like in (\ref{chave3}):}
\nor{\Psi_T(\cdot)e^{a|\cdot|}}{H_t^{1/2}}&= \nor{h_T}{H_t^{1/2}}\leq C\,T^{1/2}\nor{h}{L^2}+C\,\nor{D_t^{1/2}h}{L^2}. %\notag \\
%&\leq C\,T^{1/2}\Bigl(1+\frac{1}{T^b}\Bigr)\nor{h}{H^b_t}=C\,T^{1/2}\Bigl\langle\frac{1}{T^b}\Bigr\rangle \nor{h}{H_t^b}.
\label{chave4} \\
\intertext{We know that}
\nora{h}{L^2}{2}&=\int_{-2T}^{2T}|\Psi(t)|^2\,e^{2 a |t|T}\,dt \leq 4T\,e^{4\alpha T^2}\nora{\Psi}{L^{\infty}}{2}. \label{novouno}
\end{align}
To bounded the term $\nor{D_t^{1/2}h}{L^2}$ we are going to explore $\widehat{h}(\tau)$ integrating by parts two times,
\begin{align}
\widehat{h}(\tau)&=\int_0^{+\infty}\Psi(t)e^{aTt}e^{-it\tau}\,dt + \int_{-\infty}^0\Psi(t)e^{-aTt}e^{-it\tau}\,dt \label{novoprima}  \\
&=\frac{-1}{aT-i\tau}\,\biggl(1+\int_0^{+\infty}\frac{d\Psi}{dt}(t)\,e^{t(aT-i\tau)}\,dt\biggr) -\frac{1}{aT+i\tau}\,\biggl(1-\int_{-\infty}^0\frac{d\Psi}{dt}(t)\,e^{-t(aT+i\tau)}\,dt\biggr) \notag \\
%&=\frac{-1}{aT-i\tau}+\frac{1}{(aT-i\tau)^2}\int_0^{+\infty}\frac{d^2\Psi}{dt^2}(t)\,e^{t(aT-i\tau)}\,dt
%-\frac{1}{aT+i\tau}+\frac{1}{(aT+i\tau)^2}\int_{-\infty}^0\frac{d^2\Psi}{dt^2}(t)\,e^{-t(aT+i\tau)}\,dt \notag \\
&=\frac{-2aT}{(aT)^2+\tau^2}+\frac{1}{(aT-i\tau)^2}\int_0^{+\infty}\frac{d^2\Psi}{dt^2}(t)\,e^{t(aT-i\tau)}\,dt +\frac{1}{(aT+i\tau)^2}\int_{-\infty}^0\frac{d^2\Psi}{dt^2}(t)\,e^{-t(aT+i\tau)}\,dt . \notag %\label{novodos}
\end{align}
From this we have that
\begin{align}
|\widehat{h}(\tau)|&\leq \frac{2|a|T}{(aT)^2+\tau^2}+\frac{2(2T)\,e^{2\alpha T^2}\nor{\frac{d^2\Psi}{dt^2}}{L^{\infty}}}{(aT)^2+\tau^2}, \label{novotres} \\
\intertext{and, from (\ref{novoprima})}
|\widehat{h}(\tau)|&\leq 4T\,e^{2\alpha T^2}\nor{\Psi}{L^{\infty}}\leq 4\,e^{2 \alpha}\nor{\Psi}{L^{\infty}}=C_1\,e^{2 \alpha}. \label{novoquatro} \\
\intertext{Hence, with $C_0\,e^{2 \alpha} = 4\,e^{2\alpha }\nor{\frac{d^2\Psi}{dt^2}}{L^{\infty}}\geq 4T\,e^{2\alpha T^2}\nor{\frac{d^2\Psi}{dt^2}}{L^{\infty}}$, from (\ref{novotres}) and (\ref{novoquatro}), we obtain that}
|\widehat{h}(\tau)|&\leq \frac{2|a|T +Ce^{2\alpha} }{1+(aT)^2+\tau^2}, \label{novocinco}
\end{align}
%Because $C_0+C_1\leq 4T\,e^{2\alpha T^2}\Bigl(\nor{\frac{d^2\Psi}{dt^2}}{L^{\infty}}+\nor{\Psi}{L^{\infty}}\Bigr)\leq C\,e^{2 \alpha}$. 
where $C=C_0+C_1$. Multiplying by $|\tau|^{1/2}$ in (\ref{novocinco}), taking square and integrating on $\mathbb{R}$, we have that
\begin{align}
\nora{D_t^{1/2}h}{L^2}{2}&=\nora{|\tau|^{1/2}\widehat{h}(\tau)}{L^2}{2}\lesssim 4a^2T^2\int_{\mathbb{R}}\frac{|\tau|}{(1+a^2T^2+\tau^2)^2}\,d\tau + Ce^{4\alpha}\int_{\mathbb{R}}\frac{|\tau|}{(1+a^2T^2+\tau^2)^2}\,d\tau \notag \\
&\lesssim 4a^2T^2\int_{\mathbb{R}}\frac{|\tau|}{(a^2T^2+\tau^2)^2}\,d\tau + Ce^{4 \alpha}\int_{\mathbb{R}}\frac{|\tau|}{(1+\tau^2)^2}\,d\tau \notag \\
%&\lesssim 4a^2T^2\,\biggl(\int_{|\tau|\leq 1}\frac{d\tau}{a^2T^2}+\int_{|\tau|> 1/|a|T}\frac{|\tau|}{a^{2}T^{2}(1+\tau^2)^2}\,d\tau\biggr) %\notag \\ &\quad +Ce^{4\alpha}\,\biggl(\int_{|\tau|\leq 1}d\tau+\int_{|\tau|>1}\frac{d\tau}{|\tau|^{3}}\biggr) \notag \\
&\lesssim 4+Ce^{4\alpha} \leq C\,e^{4\alpha}, \label{novoseis}
\end{align}
where in the second inequality we used $\tau=|a|T x$. From (\ref{chave4}), (\ref{novouno}), (\ref{novoseis}) and since $T\leq 1$, we conclude (\ref{chave2}). Integrating on $\mathbb{R}$ the next inequality which is consequence of (\ref{novocinco})
\begin{equation*}
|\widehat{h}(\tau)|\lesssim  \frac{|a|T}{(aT)^2+\tau^2}+ \frac{Ce^{2\alpha}}{1+\tau^2}.
\end{equation*}
we have proved (\ref{novosiete}). The proof of (\ref{chavenew}) is equal to that of $(2.6)$ in the Lemma $2.3$ in \cite{CP}.
\end{proof}
%%%%%%%%%%%%%%%%%%%%%%%%%%%%%%%%%%%%%%%%%%%%%%%%%%%%%%%%%%%%%%%%%%%%%%%%%%%%%%%%%%%%%%%%%%%%%%%%%%%%%%%%%%%%%%%%%%%%%%%%%%%%%%%%%%%%%%%%%%%%%%%%%%%%%%%%%
%%%%%%%%%%%%%%%%%%%%%%%%%%%%%%%%%%%%%%%%%%%%%%%%%%%%%%%%%%%%%%%%%%%%%%%%%%%%%%%%%%%%%%%%%%%%%%%%%%%%%%%%%%%%%%%%%%%%%%%%%%%%%%%%%%%%%%%%%%%%%%%%%%%%%%%%%
%%%%%%%%%%%%%%%%%%%%%%%%%%%%%%%%%%%%%%%%%%%%%%%%%%%%%%%%%%%%%%%%%%%%%%%%%%%%%%%%%%%%%%%%%%%%%%%%%%%%%%%%%%%%%%%%%%%%%%%%%%%%%%%%%%%%%%%%%%%%%%%%%%%%%%%%%
%%%%%%%%%%%%%%%%%%%%%%%%%%%%%%%%%%%%%%%%%%%%%%%%%%%%%%%%%%%%%%%%%%%%%%%%%%%%%%%%%%%%%%%%%%%%%%%%%%%%%%%%%%%%%%%%%%%%%%%%%%%%%%%%%%%%%%%%%%%%%%%%%%%%%%%%%
%%%%%%%%%%%%%%%%%%%%%%%%%%%%%%%%%%%%%%%%%%%%%%%%%%%%%%%%%%%%%%%%%%%%%%%%%%%%%%%%%%%%%%%%%%%%%%%%%%%%%%%%%%%%%%%%%%%%%%%%%%%%%%%%%%%%%%%%%%%%%%%%%%%%%%%%%
%%%%%%%%%%%%%%%%%%%%%%%%%%%%%%%%%%%%%%%%%%%%%%%%%%%%%%%%%%%%%%%%%%%%%%%%%%%%%%%%%%%%%%%%%%%%%%%%%%%%%%%%%%%%%%%%%%%%%%%%%%%%%%%%%%%%%%%%%%%%%%%%%%%%%%%%%

\setcounter{equation}{0}
\section{\textsc{Linear Estimates}}

Here we study the linear operator $\Psi\,V_{\lambda}$ as well as the linear operator $M_{\lambda}$ defined as
\begin{equation}
M_{\lambda}:f \longmapsto \Psi(t)\,\int_0^tV_{\lambda}(t-t')\,f(t')\,dt'. \label{linearoperator}
\end{equation}
\subsection{\textsc{ Linear Estimates for the Free Term in $Y^s$}}
The next proposition gives a bounded to the free term of the integral equation (\ref{integralequation}).
\begin{lema}
\begin{equation}
\nor{\Psi(t)\,V_{\lambda}(t)\,\phi}{Y^s}\lesssim \nor{\phi}{H^s}.
\label{linestuno}
\end{equation}
\end{lema}
\begin{proof}[Proof]
We denote $\Theta_k(t)=\Psi(t)\,e^{\eta \,\Phi(k)\,|t|}$. Then
\begin{equation}
(\Psi(t)\,V_{\lambda}(t)\,\phi)^{\wedge}(k,\tau)=\widehat{\Theta_k(t)} \ast (e^{-(2\pi ik)^3t})^{\wedge}(\tau)\,\widehat{\phi}(k) =\widehat{\Theta_k}(\tau-4\pi^2k^3)\,\widehat{\phi}(k). \label{linestdue}
\end{equation}
So, $\nora{\Psi(t)\,V_{\lambda}(t)\,\phi}{Y^s}{2}$
\begin{align}
&\leq \nora{\langle \tau-4\pi^2k^3\rangle^{1/2}\,\langle k\rangle^s\,\widehat{\Theta_k}(\tau-4\pi^2k^3)\,\widehat{\phi}(k)}{L^2((dk)_{\lambda})L^2(d\tau)}{2}+\nora{\langle k\rangle^s\,\widehat{\Theta_k}(\tau-4\pi^2k^3)\,\widehat{\phi}(k)}{L^2((dk)_{\lambda})L^1(d\tau)}{2} \notag \\
%&=\nora{\langle\tau\rangle^{1/2}\,\widehat{\Theta}(\tau)}{L^2(d\tau)}{2}\,\nora{\langle k\rangle^s\,\widehat{\phi}(k)}{L^2((dk)_{\lambda})}{2}+\nora{\widehat{\Theta_k}(\tau)}{L^1(d\tau)}{2}\,\nora{\langle k\rangle^s\,\widehat{\phi}(k)}{L^2((dk)_{\lambda})}{2} \notag \\
&=\Bigl(\nora{\langle\tau\rangle^{1/2}\,\widehat{\Theta_k}(\tau)}{L^2(d\tau)}{2}+ \nora{\widehat{\Theta_k}(\tau)}{L^1(d\tau)}{2}\Bigr)\,\nora{\langle k\rangle^s\,\widehat{\phi}(k)}{L^2((dk)_{\lambda})}{2}. \label{linesttre}
\end{align}
(\ref{linesttre}) with (\ref{chave2}) and (\ref{novosiete}) imply (\ref{linestuno}).
\end{proof}
%%%%%%%%%%%%%%%%%%%%%%%%%%%%%%%%%%%%%%%%%%%%%%%%%%%%%%%%%%%%%%%%%%%%%%%%%%%%%%%%%%%%%%%%%%%%%%%%%%%%%%%%%%%%%%%%%%%%%%%%%%%%%%%%%%%%%%%%%%%%%%%%%%%%%%%
%%%%%%%%%%%%%%%%%%%%%%%%%%%%%%%%%%%%%%%%%%%%%%%%%%%%%%%%%%%%%%%%%%%%%%%%%%%%%%%%%%%%%%%%%%%%%%%%%%%%%%%%%%%%%%%%%%%%%%%%%%%%%%%%%%%%%%%%%%%%%%%%%%%%%%%%%
%%%%%%%%%%%%%%%%%%%%%%%%%%%%%%%%%%%%%%%%%%%%%%%%%%%%%%%%%%%%%%%%%%%%%%%%%%%%%%%%%%%%%%%%%%%%%%%%%%%%%%%%%%%%%%%%%%%%%%%%%%%%%%%%%%%%%%%%%%%%%%%%%%%%%%%%%
%%%%%%%%%%%%%%%%%%%%%%%%%%%%%%%%%%%%%%%%%%%%%%%%%%%%%%%%%%%%%%%%%%%%%%%%%%%%%%%%%%%%%%%%%%%%%%%%%%%%%%%%%%%%%%%%%%%%%%%%%%%%%%%%%%%%%%%%%%%%%%%%%%%%%%%%%
%%%%%%%%%%%%%%%%%%%%%%%%%%%%%%%%%%%%%%%%%%%%%%%%%%%%%%%%%%%%%%%%%%%%%%%%%%%%%%%%%%%%%%%%%%%%%%%%%%%%%%%%%%%%%%%%%%%%%%%%%%%%%%%%%%%%%%%%%%%%%%%%%%%%%%%%%
%\setcounter{equation}{0}
%\setcounter{section}{0}
\subsection{\textsc{ Linear Estimates for the Forcing Term in $Y^s$}}

\begin{lema}
\begin{equation}
\nor{\Psi(t)\int_0^tV_{\lambda}(t-t')\,F(t')\,dt'}{Y^s}\lesssim \nor{F}{Z^s}. \label{linestdos}
\end{equation}
\end{lema}
\begin{proof}[Proof]
By applying a smooth cutoff function, we may assume that $F$ is supported on $\mathbb{T}\times [-3,3]$. Let $a(t)=sgn(t) b(t)$, where $b$ is a smooth bump function supported on $[-10,10]$ which equals $1$ on $[-5,5]$. The identity
$$\chi_{[0,t]}(t')=\frac{1}{2}\,(a(t')+a(t-t')),$$
valid for $t\in [-2,2]$ and $t'\in [-3,3]$, allows us to rewrite
\begin{align}
\Psi(t)&\int_0^tV_{\lambda}(t-t')\,F(t')\,dt'=\Psi(t)\int_{\mathbb{R}}\chi_{[0,t]}(t')\,V_{\lambda}(t-t')\,F(t')\,dt' \notag \\
&=\frac{1}{2}\,\Psi(t)\,V_{\lambda}(t)\int_{\mathbb{R}}a(t')\,V_{\lambda}(-t')\,F(t')\,dt'+\frac{1}{2}\,\Psi(t)\int_{\mathbb{R}}a(t-t')V_{\lambda}(t-t')\,F(t')\,dt'. \label{linestone}
\end{align}
We consider the contribution of each one of the addend of (\ref{linestone}). We denote $\tilde{a}_k(t')=a(t')\,e^{\eta \Phi(k)|t'|}$ and we use (\ref{linestuno}) to obtain
\begin{align}
\nor{\Psi(t)\,V_{\lambda}(t)\int_{\mathbb{R}}a(t')\,V_{\lambda}(-t')\,F(t')\,dt'}{Y^s}&\leq \nor{\int_{\mathbb{R}}a(t')\,V_{\lambda}(-t')\,F(t')\,dt'}{H^s} \notag \\
%&=\nor{\langle k\rangle^s\int_{\mathbb{R}}a(t')\,e^{(2\pi ik)^3t'+\eta \Phi(k)|t'|}\,\widehat{F}(k,t')\,dt'}{L^2((dk)_{\lambda})} \notag \\
&=\nor{\langle k\rangle^s\int_{\mathbb{R}}e^{-2\pi i(4\pi^2 k^3)t'}\,\tilde{a}_k(t')\,\widehat{F}(k,t')\,dt'}{L^2((dk)_{\lambda})} \notag \\
%&=\nor{\langle k\rangle^s\,(\tilde{a}(t')\,\widehat{F}(k,t'))^{\wedge}(4\pi^2k^3)}{L^2((dk)_{\lambda})} \notag \\
&=\nor{\langle k\rangle^s\,\int_{\mathbb{R}}\widehat{\tilde{a}_k(t')}(4\pi^2k^3-\tau)\,\widehat{F}(k,\tau)\,d\tau}{L^2((dk)_{\lambda})} .\label{linesttwo}
\end{align}
As in the proof of \eqref{novocinco}, integrating by part twice we obtain
  $$|\widehat{\tilde{a}_k}(\tau)|\leq \dfrac{ C(\eta,\alpha)(1+ |\tau|)}{1+\eta^2 \Phi(k)^2+\tau^2}  \, \le \dfrac{C(\eta,\alpha)}{1+ |\tau|} .$$
We have from (\ref{linesttwo}) that
\begin{equation}
\nor{\Psi(t)\,V_{\lambda}(t)\int_{\mathbb{R}}a(t')\,V_{\lambda}(-t')\,F(t')\,dt'}{Y^s}\leq
C(\eta,\alpha)\,\nor{\int_{\mathbb{R}}\frac{\langle k\rangle^s\,\widehat{F}(k,\tau)}{\langle \tau-4\pi^2k^3\rangle}\,d\tau}{L^2((dk)_{\lambda})} \label{linestthree}
%&=C(\eta,\alpha)\,\nor{\frac{\langle k\rangle^s\,\widehat{F}(k,\tau)}{\langle \tau-4\pi^2k^3\rangle}}{L^2((dk)_{\lambda})L^1(d\tau)}.
\end{equation}
The contribution of the second term in (\ref{linestone}) is calculated using that multiplication by $\Psi(t)$ is a bounded operation on the space $Y^s$ (See Remark \ref{notaciondepsi}) and note that the space-time Fourier transform of $\int_{\mathbb{R}}a(t-t')\,V_{\lambda}(t-t')\,F(t')\,dt'$ is given by
\begin{align}
\Bigl(\int_{\mathbb{R}}a(t-t')\,V_{\lambda}(t-t')\,F(t')\,dt'\Bigr)^{\wedge}(k,\tau)&=\Bigl(\int_{\mathbb{R}}a(t-t')\,e^{-(2\pi ik)^3(t-t')+\eta\,\Phi(k)\,|t-t'|}\,\widehat{F}(k,t')\,dt'\Bigr)^{\wedge}(\tau) \notag \\
=\Bigl(\tilde{a}_k(\cdot)\,e^{2\pi i(4\pi^2k^3)(\cdot)}\ast \widehat{F}(k,\cdot)\,(t)\Bigr)^{\wedge}(\tau) &=\widehat{\tilde{a}_k}(\tau-4\pi^2k^3)\,\widehat{F}(k,\tau) \label{linestfour}.
\end{align}
From the definitions (\ref{normys}), (\ref{normzs}) and from the estimate for $\widehat{\tilde{a}}$ used above we have:
\begin{align}
&\nor{\int_{\mathbb{R}}a(t-t')\,V_{\lambda}(t-t')\,F(t')\,dt'}{Y^s}\notag \\
%&=\nor{\int_{\mathbb{R}}a(t-t')\,V_{\lambda}(t-t')\,F(t')\,dt'}{\mathcal{Y}_{s,1/2}} +\nor{\int_{\mathbb{R}}a(t-t')\,V_{\lambda}(t-t')\,F(t')\,dt'}{L^2((dk)_{\lambda})L^1(d\tau)} \notag \\
&=\nor{\langle \tau-4\pi^2k^3\rangle^{1/2}\,\langle k\rangle^s\,\widehat{\tilde{a}_k}(\tau-4\pi^2k^3)\,\widehat{F}(k,\tau)}{L^2((dk)_{\lambda})L^2(d\tau)}+ \nor{\langle k\rangle^s\,\widehat{\tilde{a}_k}(\tau-4\pi^2k^3)\,\widehat{F}(k,\tau)}{L^2((dk)_{\lambda})L^1(d\tau)} \notag \\
&\leq C(\eta ,\alpha)\,\nor{\langle \tau-4\pi^2k^3\rangle^{-1/2}\,\langle k\rangle^s\,\widehat{F}(k,\tau)}{L^2((dk)_{\lambda})L^2(d\tau)}+ C(\eta ,\alpha)\,\nor{\frac{\langle k\rangle^s\,\widehat{F}(k,\tau)}{\langle \tau -4\pi^2k^3\rangle}}{L^2((dk)_{\lambda})L^1(d\tau)}. \label{linestfive}
\end{align}
(\ref{linestthree}) and (\ref{linestfive}) give (\ref{linestdos}).
\end{proof}
%%%%%%%%%%%%%%%%%%%%%%%%%%%%%%%%%%%%%%%%%%%%%%%%%%%%%%%%%%%%%%%%%%%%%%%%%%%%%%%%%%%%%%%%%%%%%%%%%%%%%%%%%%%%%%%%%%%%%%%%%%%%%%%%%%%%%%%%%%%%%%%%%%%%%%%%%%
%%%%%%%%%%%%%%%%%%%%%%%%%%%%%%%%%%%%%%%%%%%%%%%%%%%%%%%%%%%%%%%%%%%%%%%%%%%%%%%%%%%%%%%%%%%%%%%%%%%%%%%%%%%%%%%%%%%%%%%%%%%%%%%%%%%%%%%%%%%%%%%%%%%%%%%%%%
%%%%%%%%%%%%%%%%%%%%%%%%%%%%%%%%%%%%%%%%%%%%%%%%%%%%%%%%%%%%%%%%%%%%%%%%%%%%%%%%%%%%%%%%%%%%%%%%%%%%%%%%%%%%%%%%%%%%%%%%%%%%%%%%%%%%%%%%%%%%%%%%%%%%%%%
\subsection{\textsc{Linear Estimates for the Forcing Term in $\mathcal{Y}_{s,1/2}$}}
%Now, we are going to obtain a estimate of the forcing term in the space $\mathcal{Y}_{s,1/2}$
\begin{prop}\label{forcingterm}
Let $T\in [\frac{\sqrt{2}}{\sqrt{\beta}},\frac{1}{2}]$, $s\in \mathbb{R}$ and $\beta>8$. Then,
\begin{align}
\nor{\Psi_T(t)\int_0^t V_{\lambda}(t-t')\,F(t')\,dt'}{\mathcal{Y}_{s,1/2}}\leq C\,\eta\,\alpha\,(\beta+(\eta\,\alpha)^2)\,e^{2\eta\,\alpha}\,T^{1/2}\,\nor{F}{\mathcal{Y}_{s,-1/2}}. \label{cotapartenaolineal}
\end{align}
where $C$ is a constant.
\end{prop}
\begin{rem}
The Proposition \ref{forcingterm} together with the inequality (\ref{chave2}), which implies a linear estimate for the free term in $\mathcal{Y}_{s,1/2}$, that is, $\nor{\Psi(t)\,V_{\lambda}(t)\,\phi}{\mathcal{Y}_{s,1/2}}\leq C\,e^{2\alpha}\,\nor{\phi}{H^s}$, and the bilinear estimate from Kenig, Ponce and Vega given in the Proposition \ref{bilest} guarantee the local well-posedness result to the $\lambda$-periodic initial value problem (\ref{fivp}) in the Sobolev spaces $H^s(\mathbb{T})$ to $s>-1/2$ at least for small initial data.
\end{rem}
To prove this Proposition \ref{forcingterm} we need the next Lemmas:
\begin{lema}[Schur's Lemma]\label{lemadeschur}
Let $f$ be in $\mathcal{S}(\mathbb{R})$ and $L$ the integral operator, given by
\begin{align*}
(Lf)(x)=\int_{\mathbb{R}}N(x,y)\,f(y)\,dy
\end{align*}
where the kernel $N$ is such that
\begin{equation*}
\sup_x \; \int_{\mathbb{R}}|N(x,y)|\,dy \leq 1,\qquad \text{and} \qquad
\sup_y \; \int_{\mathbb{R}}|N(x,y)|\,dx \leq 1.
\end{equation*}
Then, $\nor{L}{L^2\to L^2}\leq 1$.
\end{lema}
\begin{proof}[Proof]
See the section $2.4.1$, page $284$ of \cite{S}.
\end{proof}
\begin{lema}\label{keylema}
Let $\frac{\sqrt{2}}{\sqrt{\beta}}\leq T\leq 1$, $\beta \geq 2$, $\alpha\geq 1$ and $|a|\leq \alpha$. Then
\begin{align}\label{thankgod}
\nor{\Psi_T(\cdot)\,I_a(\cdot)}{H_t^{1/2}} &\leq
C\,\alpha \,(\beta\,+\alpha^2)\,e^{2\alpha }\,T^{1/2}\,\nor{f}{H^{-1/2}},
\end{align}
where $I_a(t):= \int_0^t e^{a|t-t'|}f(t')\,dt'$.
\end{lema}
\begin{proof}[Proof]
Rewrite $I_a(t)$ as in the proof of the Lemma 2.4 in \cite{CP}. By Fourier inverse transform, we have $I_a(t)= \int_{\mathbb{R}} \widehat{f}(\tau) \dfrac{e^{i\tau t}-e^{a|t|}}{i\tau - sgn(t)a}\,d\tau $. Since, $\dfrac{1}{sgn(t)a - i\tau } = sgn(t) p_a(\tau) + i q_a(\tau)$ where $p_a(\tau)=\dfrac{a}{a^2+\tau^2}$ and $q_a(t)=\dfrac{\tau}{a^2+\tau^2}$ then, replacing $\tau $ by $t'$, we obtain
\begin{align}
I_a(t) = sgn(t) \int_{\mathbb{R}} p_a(t')[e^{a|t|}-e^{it't}]\widehat{f}(t')\,dt' +i\int_{\mathbb{R}} q_a(t')[e^{a|t|}-e^{it't}]\widehat{f}(t')\,dt' := I_{a,1}(t)+ I_{a,2}(t).     \label{thankone}
\end{align}
%%%%%%%%%%%%%%%%%%%%%%%%%%%%%%%%%%%%%%%%%%%%%%%%%%%%%%%%%%%%%%%%%%%%%%%%%%%%%%%%%%%%%%%%%%%%%%%%%%%%%%%%%%%%%%%%%%%%%%%%%%%%%%%%%%%%%%%%%%%%%%%%%%%%%%%
%%%%%%%%%%%%%%%%%%%%%%%%%%%%%%%%%%%%%%%%%%%%%%%%%%%%%%%%%%%%%%%%%%%%%%%%%%%%%%%%%%%%%%%%%%%%%%%%%%%%%%%%%%%%%%%%%%%%%%%%%%%%%%%%%%%%%%%%%%%%%%%%%%%%%%%
\textbf{Estimate for $I_{a,1}$}. We write
\begin{align}
I_{a,1}(t)&=sgn(t) \int_{|t'|>1/T} p_a(t')[e^{a|t|}-e^{it't}]\widehat{f}(t')\,dt' + sgn(t) \int_{|t'|\leq 1/T} p_a(t')[e^{a|t|}-e^{it't}]\widehat{f}(t')\,dt' \notag \\
&:=I^>_{a,1}(t)+I^<_{a,1}(t)\label{thanktwo}
\end{align}

\textbf{Case 1:} $|t'|>1/T$. In this case $|t'|\backsimeq \langle t' \rangle$.
\begin{align}
\Psi_T(t)\,I^>_{a,1}(t)=\Psi_T(t)sgn(t) \int_{|t'|>1/T} p_a(t')[e^{a|t|}-e^{it't}]\widehat{f}(t')\,dt' = a h_T(t), \label{thankthree}
\end{align}
where $h_T(t)=h(t/T)$ and
\begin{align}
h(t)=\Psi(t)sgn(t) \int_{|t'|>1/T} \dfrac{\widehat{f}(t')}{a^2+(t')^2}[e^{aT|t|}-e^{iTt't}]\,dt' ,\label{thankfour}
\end{align}
\begin{align}
\widehat{h(t)}(\tau)=\int_{|t'|>1/T} \dfrac{\widehat{f}(t')}{a^2+(t')^2} K(a,T,\tau,t')\,dt' \label{thankfive}
\end{align}
with
\begin{equation}
K(a,T,\tau,t')=\int_{\mathbb{R}} sgn(t) \,\Psi(t)[e^{aT|t|}-e^{iTt't}]\,e^{-it\tau}\,dt. \label{thanksix}
\end{equation}
Integrating by parts, we have
\begin{align}
K(a,T,\tau,t')&=\int_{\mathbb{R}} sgn(t) \Psi(t)\,e^{aT|t|}\,e^{-it\tau}\,dt - \int_{\mathbb{R}} sgn(t) \Psi(t)\,e^{iTt't}\,e^{-it\tau}\,dt \notag \\
&=-\dfrac{2\,i}{\tau}-\dfrac{i}{\tau}\int_{\mathbb{R}}sgn(t)\Bigl(\dfrac{d}{dt}\Psi(t)+aT sgn(t)\,\Psi(t)\Bigr)\,e^{aT|t|}\,e^{-it \tau}\,dt + \notag \\
&\quad \,+\dfrac{2\,i}{\tau} +\dfrac{i}{\tau}\int_{\mathbb{R}}sgn(t)\,\Bigl(\dfrac{d}{dt}\Psi(t)+\,iTt'\,\Psi(t)\Bigr)\,e^{i(Tt'-\tau)t}\,dt  \notag \\
&=K_1(a,T,\tau)+K_2(T,\tau,t'), \label{quebraka}
\end{align}
where
\begin{align}
K_1(a,T,\tau)&=-\dfrac{i}{\tau}\int_{\mathbb{R}}sgn(t)\Bigl(\dfrac{d}{dt}\Psi(t)+aT sgn(t)\,\Psi(t)\Bigr)\,e^{aT|t|}\,e^{-it \tau}\,dt \notag \\
&=-\dfrac{1}{\tau^2}\int_{\mathbb{R}}sgn(t)\Bigl(\dfrac{d^2}{dt^2}\Psi(t)+2aT\,sgn(t)\,\dfrac{d}{dt}\Psi(t)+\,(aT)^2\,\Psi(t)\Bigr)e^{aT|t|}\,e^{-it \tau}\,dt ,\label{kasubuno} \\
|K_1(a,T,\tau)|&\leq \dfrac{1}{|\tau|^2}\int_{-2}^2\Bigl(\Bigl|\dfrac{d^2}{dt^2}\Psi(t)\Bigr|+2|a|T\,\Bigl|\dfrac{d}{dt}\Psi(t)\Bigr|+\,(|a|T)^2\,|\Psi(t)|\Bigr)e^{aT|t|}\,dt
\leq \dfrac{C\,(1+\alpha T )^2\,e^{2a}}{|\tau|^2}, \label{cotakasubuno} \\
\intertext{and}
K_2(T,\tau,t')&=\dfrac{i}{\tau}\int_{\mathbb{R}}sgn(t)\,\Bigl(\dfrac{d}{dt}\Psi(t)+\,iTt'\,\Psi(t)\Bigr)\,e^{i(Tt'-\tau)t}\,dt  \label{kasubdosone} \\
&=\dfrac{2iTt'}{\tau^2}-\dfrac{1}{\tau^2}\int_{\mathbb{R}}sgn(t)\,\Bigl(\dfrac{d^2\Psi(t)}{dt^2}+2iTt'\,\dfrac{d\Psi(t)}{dt}+(iTt')^2\,\Psi(t)\Bigr)\,e^{i(Tt'-\tau)t}\,dt  \label{kasubdostwo} \\
&=\dfrac{1}{\tau (Tt'-\tau)}\int_{\mathbb{R}}sgn(t)\,\,\dfrac{d^2\Psi(t)}{dt^2}\,e^{it(Tt'-\tau)}\,dt +\notag \\
&- \dfrac{iTt'}{\tau (Tt'-\tau)}\int_{\mathbb{R}}sgn(t)\,\,\dfrac{d\Psi(t)}{dt}\,e^{it(Tt'-\tau)}\,dt -\dfrac{2iTt'}{\tau (Tt'-\tau)}. \label{kasubdosthree}
\end{align}
Thus, from (\ref{kasubdosone}):
\begin{align}
|K_2(T,\tau,t')|&\leq C\,\dfrac{|t'|}{|\tau|},  \label{cotakasubdosone} \\
\intertext{from (\ref{kasubdostwo}):}
|K_2(T,\tau,t')|&\leq C\,\dfrac{|t'|^2}{|\tau|^2}, \label{cotakasubdostwo}
\end{align}
and from (\ref{kasubdosthree}):
\begin{align}
|K_2(T,\tau,t')|&\leq \dfrac{1}{|\tau (Tt'-\tau)|}\int_{-2}^2\Bigl|\dfrac{d^2\Psi(t)}{dt^2}\Bigr|\,dt + \dfrac{T|t'|}{|\tau (Tt'-\tau)|}\int_{-2}^2\Bigl|\dfrac{d\Psi(t)}{dt}\Bigr|\,dt +\dfrac{2T|t'|}{|\tau (Tt'-\tau)|} \notag \\
&\leq \dfrac{C\,T|t'|}{|\tau (Tt'-\tau)|}. \label{cotakasubdosthree}
\end{align}
In the inequalities above $C=C_{\Psi}=\max\{\nor{\Psi}{L^{\infty}}, \nor{\frac{d}{dt}\Psi}{L^{\infty}}, \nor{\frac{d^2}{dt^2}\Psi}{L^{\infty}}\}$. From (\ref{thankfive}), (\ref{quebraka}), (\ref{cotakasubuno}) and considering $|\tau|> 1/2$, which implies $\langle \tau \rangle \backsimeq |\tau|$, we have
\begin{align}
|&\widehat{h(t)}(\tau)|\leq \int_{|t'|>1/T} \dfrac{|\widehat{f}(t')|}{a^2+(t')^2} |K_1(a,T,\tau)|\,dt' + \int_{|t'|>1/T} \dfrac{|\widehat{f}(t')|}{a^2+(t')^2} |K_2(T,\tau,t')|\,dt' \notag \\
&\leq \int_{|t'|>1/T} \dfrac{|\widehat{f}(t')|}{[a^2+(t')^2]}\, C\,\dfrac{(1+\alpha T)^2\,e^{2\,a}}{|\tau|^2}\,dt' + \int_{1/T<|t'|\leq \beta |\tau|T} \dfrac{|\widehat{f}(t')|}{a^2+(t')^2} |K_2(T,\tau,t')|\,dt' + \notag \\
&+ \int_{|t'|\geq \beta |\tau|T } \dfrac{|\widehat{f}(t')|}{a^2+(t')^2}\, |K_2(T,\tau,t')|\,dt' \notag \\
&= J_1+J_2+J_3 .\label{veruno}
\end{align}
We obtained $J_2$ and $J_3$ splitting the set $\{|t'|>1/T\}$ in $\{1/T<|t'|< \beta |\tau| T\}\ne \{\}$ (because $T\geq \sqrt{2}/\sqrt{\beta}$) and $\{|t'|\geq \beta |\tau| T\}$ where $\beta \geq 2$. We estimate $J_1$ so,
\begin{align}
J_1&\leq C\,\dfrac{(1+\alpha T)^2\,e^{2\,a}}{|\tau|^2}\, \int_{|t'|>1/T} \dfrac{|\widehat{f}(t')|}{\langle t' \rangle^{1/2}}\; \dfrac{\langle t' \rangle^{1/2}}{(t')^2}\,dt' \notag \\
&\leq C\,\dfrac{(1+\alpha T)^2\,e^{2\,a}}{|\tau|^2}\, \nor{f}{H^{-1/2}}\Bigl(\int_{|t'|>1/T} \dfrac{1}{(t')^3}\,dt' \Bigr)^{1/2}\notag \\
&\leq C\,\dfrac{(1+\alpha T)^2\,e^{2\,a}\,T}{|\tau|^2}\, \nor{f}{H^{-1/2}}. \label{verdue}
\end{align}
From (\ref{cotakasubdosthree}) we obtain
\begin{equation}
|K_2(T,\tau,t')|\leq \dfrac{C\,T|t'|}{|\tau (Tt'-\tau)|} \leq \dfrac{2^{1-\gamma}C\,T|t'|}{|\tau |\,|\tau|^{\gamma}\,|Tt'|^{1-\gamma}} \quad \text{always that}\;\;|t'|\geq 2\frac{|\tau |}{T} \label{cotakasubdosfour}
\end{equation}
because $|Tt'-\tau|\geq |Tt'|-|\tau | \geq |\tau|$, $|Tt'-\tau|\geq |Tt'|-|\tau | \geq |Tt'|/2$ and so, for $0\leq \gamma \leq 1$
$$|Tt'-\tau|\geq |\tau|^{\gamma}\,\dfrac{|Tt'|^{1-\gamma}}{2^{1-\gamma}}.$$
Note that $1/T< 2|\tau|/T\leq \beta |\tau | T $. So, to estimate the integral $J_3$, since $T\geq \frac{\sqrt{2}}{\sqrt{\beta}}$, then $\frac{2|\tau|}{T}\leq \beta |\tau| T \leq |t'|$, hence
\begin{align}
J_3 &\leq C\,\int_{|t'|\geq \beta |\tau |T} \dfrac{|\widehat{f}(t')|}{\langle t' \rangle^{1/2}}\,\dfrac{\langle t'\rangle^{1/2}\,T\,|t'|}{|\tau |^{1+\gamma}\,T^{1-\gamma}\,|t'|^{3-\gamma}}\,dt' \notag \\
&\leq \dfrac{C\,T^{\gamma}}{|\tau|^{\gamma +1}}\,\nor{f}{H^{-1/2}}\,\Bigl(\int_{|t'|>1/T} \dfrac{1}{|t'|^{3-2\gamma}}\,dt' \Bigr)^{1/2}\notag \\
&\leq \dfrac{C\,T}{|\tau|^{\gamma +1}}\, \nor{f}{H^{-1/2}}.  \label{vertre}
\end{align}
To estimate $J_2$ we are going to use the Schur's lemma \ref{lemadeschur}
\begin{align}
J_2&\leq \int_{1/T<|t'|\leq \beta |\tau|T} \dfrac{|\widehat{f}(t')|}{(t')^2} |K_2(T,\tau,t')|\,dt' \notag \\
&\leq \dfrac{1}{|\tau|^{1/2}}\int_{1/T<|t'|\leq \beta |\tau|T} \dfrac{|\widehat{f}(t')|}{|t'|^{1/2}}\;\dfrac{|\tau|^{1/2}|K_2(T,\tau,t')|}{|t'|^{3/2}}\,dt' ,\label{verquatro}
\end{align}
from $\langle \tau \rangle \backsimeq |\tau|$ and from (\ref{veruno}) we have
\begin{align}
|\widehat{h(t)}(\tau)|&\leq C\,\Bigl(\dfrac{(1+\,\alpha T)^2\,e^{2a}\,T}{\langle\tau \rangle^2}+\dfrac{T}{\langle\tau \rangle^{1+\gamma}}\Bigr)\,\nor{f}{H^{-1/2}} \notag \\
&+ \dfrac{C}{\langle \tau \rangle^{1/2}}\int_{1/T<|t'|\leq \beta |\tau|T} \dfrac{|\widehat{f}(t')|}{\langle t' \rangle^{1/2}}\;\dfrac{|\tau|^{1/2}|K_2(T,\tau,t')|}{|t'|^{3/2}}\,dt'. \label{vercinque}
\end{align}
Now, we consider the integral operator $(L_Tg)(\tau)=\int_{\mathbb{R}}N_T(\tau,t')\,g(t')\,dt'$ where $g(t')=\dfrac{|\widehat{f}(t')|}{\langle t' \rangle^{1/2}}$ and $N_T(\tau,t')=\dfrac{|\tau|^{1/2}}{|t'|^{3/2}}\;|K_2(T,\tau,t')|\;\chi_{\{\frac{1}{T}<|t'|\leq \beta |\tau|T\}}$. So, we obtain from (\ref{vercinque}) that
\begin{align}
|\widehat{h(t)}(\tau)|\leq C\,\Bigl(\dfrac{(1+\,\alpha T)^2\,e^{2a}\,T}{\langle\tau \rangle^2}+\dfrac{T}{\langle\tau \rangle^{1+\gamma}}\Bigr)\,\nor{f}{H^{-1/2}} + \dfrac{C}{\langle \tau \rangle^{1/2}}\,L_Tg(\tau), \label{versei}
\end{align}
We multiply (\ref{versei}) by $\langle \tau \rangle^{1/2}$, take the $L^2$ norm and obtain
\begin{align}
\nor{\langle \tau \rangle^{1/2}\widehat{h(t)}(\tau)}{L^2}&\leq C\,\Bigl(\sqrt{\frac{\beta}{2}}+\,\alpha\Bigr)^2\,T^3\,e^{2a}\,\nor{f}{H^{-1/2}}\,\nor{\dfrac{1}{\langle \tau \rangle^{3/2}}}{L^2}  \notag \\ &+C\,T\,\nor{f}{H^{-1/2}}\,\nor{\dfrac{1}{\langle \tau \rangle^{1/2+\gamma}}}{L^2}+C\, \nor{L_Tg(\tau)}{L^2}. \label{versete}
\end{align}
It is sufficient to prove that the operator $L_T$ is bounded in $L^2$. We need to prove that
\begin{equation*}
\sup_{\tau}\int_{\mathbb{R}}|N_T(\tau,t')|\,dt' \leq C(T) \qquad \text{and} \qquad \sup_{t'}\int_{\mathbb{R}}|N_T(\tau,t')|\,d\tau \leq C(T)
\end{equation*}
to apply the Schur's Lemma \ref{lemadeschur}. We proceed using (\ref{cotakasubdosone})
\begin{align}
\sup_{\tau}\int_{\mathbb{R}}|N_T(\tau,t')|\,dt' &\leq C\, \sup_{\tau} \int_{1/T<|t'|\leq \beta |\tau|T}\dfrac{|\tau|^{1/2}}{|t'|^{3/2}}\;\dfrac{|t'|}{|\tau|}\,dt' \leq C\, \sup_{\tau} \int_{0}^{\beta |\tau|T}\dfrac{|\tau|^{-1/2}}{|t'|^{1/2}}\,dt' \notag \\
&\leq C\, \sup_{\tau} |\tau|^{-1/2}\,T^{1/2}\,(\beta |\tau|)^{1/2} = C\,\beta^{1/2}\,T^{1/2}, \label{supentau}
\end{align}
and using (\ref{cotakasubdostwo})
\begin{align}
\sup_{t'}\int_{\mathbb{R}}|N_T(\tau,t')|\,d\tau  &\leq C\, \sup_{t'} \int_{|\tau|> 1/2}\dfrac{|\tau|^{1/2}}{|t'|^{3/2}}\;\dfrac{|t'|^2}{|\tau|^2}\,d\tau \leq C\,\sup_{t'} \int_{|t'|/\beta T}^{+\infty}\dfrac{|t'|^{1/2}}{|\tau|^{3/2}}\,d\tau \notag \\
&\leq C\, \sup_{t'} \dfrac{|t'|^{1/2}}{1/2}\,\biggl(\frac{|t'|}{\beta T}\biggr)^{-1/2} = C\,\beta^{1/2}\,T^{1/2}. \label{supentlinha}
\end{align}
Hence, $\nor{L_Tg}{L^2}\leq C\, \beta^{1/2}\,T^{1/2}$. Then, from (\ref{versete})
\begin{align}
\nor{\langle \tau \rangle^{1/2}\widehat{h(t)}(\tau)}{L^2}&\leq C\,\Bigl[\Bigl(\sqrt{\frac{\beta}{2}}+\,\alpha\Bigr)^2\,T^3\,e^{2a}+T\Bigr]\,\nor{f}{H^{-1/2}}  + C\,\beta^{1/2}\,T^{1/2}\,\nor{g}{L^2} \notag \\
\nor{h}{H^{1/2}}&\leq C\,\Bigl[\Bigl(\sqrt{\frac{\beta}{2}}+\,\alpha\Bigr)^2\,T^3\,e^{2a}+T+\beta^{1/2}\,T^{1/2}\Bigr]\,\nor{f}{H^{-1/2}} \notag \\
&\leq C\,(\beta +\alpha^2)\,e^{2a}\,T^{1/2}\,\nor{f}{H^{-1/2}} \label{verotto}
\end{align}
when $|\tau|> 1/2$. If $|\tau|\leq 1/2$ we use that $|K(a,T,\tau , t')|\leq C\,(e^{2a}+1)$ which is consequence of (\ref{thanksix}). So,
\begin{align}
|\widehat{h(t)}(\tau)|&\leq C\,(e^{2a}+1)\,\int_{|t'|>1/T}\dfrac{|\widehat{f}(t')|}{a^2+(t')^2}\,dt'=C\,(e^{2a}+1)\,\int_{|t'|>1/T}\dfrac{|\widehat{f}(t')|}{\langle t'\rangle^{1/2}}\;\dfrac{|t'|^{1/2}}{|t'|^2}\,dt' \notag \\
&\leq C\,(e^{2a}+1)\,\nor{f}{H^{-1/2}}\,\Bigl(\int_{|t'|>1/T}\dfrac{dt'}{|t'|^3}\Bigr)^{1/2}=C\,(e^{2a}+1)\,T\,\nor{f}{H^{-1/2}}, \notag \\
\intertext{and,}
\int_{|\tau|\leq 1/2}\langle \tau \rangle\,|\widehat{h}(\tau)|^2\,d\tau &\leq C\,(e^{2a}+1)\,T\,\nor{f}{H^{-1/2}}\,\Bigl(\int_{|\tau|\leq 1/2} \langle \tau \rangle \,d\tau \Bigr)^{1/2}= C\,(e^{2a}+1)\,T\,\nor{f}{H^{-1/2}}. \label{thankseven}
\end{align}
Thus, adding (\ref{verotto}) and (\ref{thankseven}), we have that
\begin{align}
\nor{h}{H^{1/2}} &\leq C\,[(\beta +\alpha^2)\,e^{2a}\,T^{1/2}+(e^{2a}+1)\,T]\,\nor{f}{H^{-1/2}} \notag \\
&\leq C\,[(\beta +\alpha^2)\,e^{2a}+1]\,T^{1/2}\,\nor{f}{H^{-1/2}}. \label{cotadeache}
\end{align}
Finally, from (\ref{thankthree}) and (\ref{cotadeache}),
\begin{align}
\nor{\Psi_TI^>_{a,1}}{H^{1/2}}&=\nor{ah_T}{H^{1/2}}\leq C\,|a|\,(T^{1/2}+1)\, \nor{h}{H^{1/2}} \notag \\
&\leq C\,\alpha \, (T^{1/2}+1) \,[(\beta +\alpha^2)\,e^{2a}+1]\,T^{1/2}\,\nor{f}{H^{-1/2}} \notag \\
&\leq C\,\alpha \,[(\beta +\alpha^2)\,e^{2a}+1]\,T^{1/2}\,\nor{f}{H^{-1/2}}, \label{inespeuno}
\end{align}
and (\ref{thankgod}) is proved in this case. \\ \\
%%%%%%%%%%%%%%%%%%%%%%%%%%%%%%%%%%%%%%%%%%%%%%%%%%%%%%%%%%%%%%%%%%%%%%%%%%%%%%%%%%%%%%%%%%%%%%%%%%%%%%%%%%%%%%%%%%%%%%%%%%%%%%%%%%%%%%%%%%%%%%%%%%%%%%%
%%%%%%%%%%%%%%%%%%%%%%%%%%%%%%%%%%%%%%%%%%%%%%%%%%%%%%%%%%%%%%%%%%%%%%%%%%%%%%%%%%%%%%%%%%%%%%%%%%%%%%%%%%%%%%%%%%%%%%%%%%%%%%%%%%%%%%%%%%%%%%%%%%%%%%%
\textbf{Case 2:} $|t'|\leq 1/T$. We proceed like in \cite{CP}, so $\widetilde{\Psi_T}(t)= sgn(t)\Psi_T(t)$ and
\begin{align}
(\Psi_T(t)I^<_{a,1}(t))^{\wedge}(\tau)&= \int_{|t'|\leq 1/T}p_a(t')\widehat{f}(t')\{(\widetilde{\Psi_T}(t)e^{a|t|})^{\wedge}(\tau)-(\widetilde{\Psi_T}(t)e^{a|t|})^{\wedge}(\tau -t')\}\,dt' + \notag \\
& \quad + \int_{|t'|\leq 1/T}p_a(t')\widehat{f}(t')(\widetilde{\Psi_T}(t)[e^{a|t|}-1])^{\wedge}(\tau-t')\,dt' \notag \\
&:= I_{a,11}(\tau) + I_{a,12}(\tau). \label{dosuno}
\end{align}
We can estimate the integral $I_{a,11}$ with the ideas used to prove the Lemma $2.1$ in \cite{GTV}.
\begin{align}
I_{a,11}&=\int_{|t'|\leq 1/T}p_a(t')\,\widehat{f}(t')\int_{\tau -t'}^{\tau}\frac{d}{du}(\widetilde{\Psi_T}(t)e^{a|t|})^{\wedge}(u)\,du \;dt' \notag \\
&=\int_{|t'|\leq 1/T}\dfrac{a\,t'}{a^2+(t')^2}\,\widehat{f}(t')\int_0^1\frac{d}{d\lambda}(\widetilde{\Psi_T}(t)e^{a|t|})^{\wedge}(\tau -\lambda t')\,d\lambda \;dt'. \label{aruno}
\end{align}
We multiply (\ref{aruno}) by $\langle \tau \rangle^{1/2}\leq C(\langle t' \rangle^{1/2}+|\tau -\lambda t'|^{1/2})$, take the $L^2$ norm and obtain
\begin{align}
\nor{I_{a,11}}{H^{1/2}}&\leq C \int_{|t'|\leq 1/T}\dfrac{|a|\,|t'|\,\langle t' \rangle^{1/2}\,|\widehat{f}(t')|}{a^2+(t')^2}\,dt'\;\nor{\dfrac{d}{dt}(\widetilde{\Psi_T}e^{a|\cdot|})^{\wedge}}{L_{\tau}^2}+ \notag \\
&\quad +C\int_{|t'|\leq 1/T}\dfrac{|a|\,|t'|\,|\widehat{f}(t')|}{a^2+(t')^2}\,dt'\;\nor{|\tau|^{1/2}\dfrac{d}{dt}(\widetilde{\Psi_T}e^{a|\cdot|})^{\wedge}}{L_{\tau}^2} \notag \\
&\leq C \int_{|t'|\leq 1/T}\dfrac{|\widehat{f}(t')|}{\langle t' \rangle^{1/2}}\,\dfrac{|a|\,|t'|\,\langle t' \rangle}{a^2+(t')^2}\,dt'\;\nor{(|t|\,\Psi_T(t)e^{a|t|})^{\wedge}(\tau)}{L^2_{\tau}}+ \notag \\
&\quad +C\int_{|t'|\leq 1/T}\dfrac{|\widehat{f}(t')|}{\langle t' \rangle^{1/2}}\,\dfrac{|a|\,|t'|\,\langle t' \rangle^{1/2}}{a^2+(t')^2}\,dt'\;\nor{|\tau |^{1/2}(|t|\,\Psi_T(t)e^{a|t|})^{\wedge}(\tau)}{L^2_{\tau}} \notag \\
&\leq C\,T^{3/2}\,\nor{f}{H^{-1/2}}\Bigl( \int_{|t'|\leq 1/T}\dfrac{|a|^2\,|t'|^2\,\langle t' \rangle^2}{[a^2+(t')^2]^2}\,dt'\Bigr)^{1/2}+\notag \\
&\quad +C\,T\,\nor{f}{H^{-1/2}}\Bigl(\int_{|t'|\leq 1/T}\dfrac{|a|^2\,|t'|^2\,\langle t' \rangle}{[a^2+(t')^2]^2}\,dt'\Bigr)^{1/2} \label{ardos} \\
&\leq C\,\alpha \,T\,\nor{f}{H^{-1/2}}+C\,\alpha \,T^{1/2}\,\nor{f}{H^{-1/2}} \label{artres} \\
&\leq C\,\alpha \,T^{1/2}\,\nor{f}{H^{-1/2}}. \label{arcuatro}
\end{align}
We obtained (\ref{ardos}) thanks to Cauchy-Schwartz's inequality and (\ref{chavenew}). (\ref{artres}) is consequence from
$$|a|^2|t'|^4\leq \alpha^2\,|t'|^4 \leq \alpha^2\,[a^2+(t')^2]^2$$
and this implies that the root square of the integrals in (\ref{ardos}) are bounded by $\sqrt{2}\,\alpha \,T^{-1/2}$.
\\ \\
The estimate of the integral $I_{a,12}$ in \cite{CP} is not adequate but a small modification is sufficient to obtain a good result. From Case 2 in the proof of the Lemma 2.4 in \cite{CP} we know that
\begin{align}
|(\widetilde{\Psi_T}(t)[e^{a|t|}-1])^{\wedge}(\tau-t')|\leq C\,T^2\;\dfrac{|a|}{(1+|\tau|T)^2}. \label{dosdos}
\end{align}
So, using Cauchy-Schwartz's inequality
\begin{align}
|I_{a,12}(\tau)|&\leq C \int_{|t'|\leq 1/T}\dfrac{|a|^2}{a^2+(t')^2}\,|\widehat{f}(t')|\;\dfrac{T^2}{(1+|\tau|T)^2}\,dt' \notag \\
&= \dfrac{CT^2}{(1+|\tau|T)^2} \int_{|t'|\leq 1/T}\dfrac{|\widehat{f}(t')|}{\langle t' \rangle^{1/2}}\;\dfrac{|a|^2\langle t' \rangle^{1/2}}{a^2+(t')^2}\,dt' \notag \\
&\leq \dfrac{CT^2}{(1+|\tau|T)^2} \Bigl(\int_{|t'|\leq 1/T}\dfrac{|a|^4\langle t' \rangle}{[a^2+(t')^2]^2}\,dt'\Bigr)^{1/2}\nor{f}{H^{-1/2}}. \label{dostres}
\end{align}
Taking square, multiplying by $\langle \tau \rangle$ and integrating on $\mathbb{R}$ in (\ref{dostres}) we obtain
\begin{align}
\int_{\mathbb{R}}\langle \tau \rangle|I_{a,12}(\tau)|^2\,d\tau &\leq CT^4\Bigl\{\int_{\mathbb{R}}\dfrac{(1+|\tau|)}{(1+|\tau|T)^4}\,d\tau \Bigr\}\Bigl\{\int_{|t'|\leq 1/T}\dfrac{|a|^4(1+|t'|)}{[|a|+|t'|]^4}\,dt'\Bigr\}\nora{f}{H^{-1/2}}{2} \notag \\
&\leq C\,(\alpha +1)\,T\,\nora{f}{H^{-1/2}}{2} \label{doscuatro}
\end{align}
because
\begin{align}
\int_{\mathbb{R}}\dfrac{(1+|\tau|)}{(1+|\tau|T)^4}\,d\tau &\leq  C\,\Bigl(\dfrac{1}{T}+\dfrac{1}{T^2}\Bigr) \notag \\
\intertext{and}
\int_{|t'|\leq 1/T}\dfrac{|a|^4(1+|t'|)}{[|a|+|t'|]^4}\,dt' &\leq \dfrac{1+|a|}{T}\leq \dfrac{\alpha +1}{T}. \notag
\end{align}
Hence, from (\ref{doscuatro}),
\begin{align}
\nor{I_{a,12}}{H^{1/2}}&\leq C\,\sqrt{\alpha +1}\,T^{1/2}\,\nor{f}{H^{-1/2}}. \label{doscinco} \\
\intertext{We conclude from (\ref{dosuno}), (\ref{arcuatro}) and (\ref{doscinco}) that}
\nor{\Psi_TI^<_{a,1}}{H^{1/2}}&\leq C\,(\alpha +\sqrt{\alpha +1}\,)\,T^{1/2}\,\nor{f}{H^{-1/2}}. \label{inespedos}
\end{align}
Note that, from (\ref{inespeuno}) and (\ref{inespedos}), we have actually proved that
\begin{equation}
\nor{\Psi_TI_{a,1}}{H^{1/2}}\leq C\,\{\,\alpha\,[(\beta +\alpha^2)\,e^{2a}+1]+ \sqrt{\alpha +1}\}\,T^{1/2}\,\nor{f}{H^{-1/2}}, \label{inespetres}
\end{equation}
which gives (\ref{thankgod}).\\ \\
\textbf{Estimate for $I_{a,2}$.} The estimate for $I_{a,2}$ is similar to that of $I_{a,1}$, exchanging $p_a$ by $q_a$ and $\widetilde{\Psi_T}$ by $\Psi_T$. So, we omit its calculation.
\end{proof}
\begin{corol}\label{salva}
Let $\frac{\sqrt{2}}{\sqrt{\beta}}\leq T\leq 1$, $\beta\geq 2$, $\alpha\geq 1$ and $-\alpha\leq a\leq 0$.
\begin{equation}
\nor{\Psi_TI_a}{H^{1/2}} \leq C\,\alpha\,(\beta+\alpha^2)\,T^{1/2}\,\nor{f}{H^{-1/2}}. \label{salvador}
\end{equation}
\end{corol}
\begin{proof}[Proof]
(\ref{salvador}) is a direct consequence of (\ref{inespetres}).
\end{proof}
%%%%%%%%%%%%%%%%%%%%%%%%%%%%%%%%%%%%%%%%%%%%%%%%%%%%%%%%%%%%%%%%%%%%%%%%%%%%%%%%%%%%%%%%%%%%%%%%%%%%%%%%%%%%%%%%%%%%%%%%%%%%%%%%%%%%%%%%%%%%%%%%%%%%%%%%%
%%%%%%%%%%%%%%%%%%%%%%%%%%%%%%%%%%%%%%%%%%%%%%%%%%%%%%%%%%%%%%%%%%%%%%%%%%%%%%%%%%%%%%%%%%%%%%%%%%%%%%%%%%%%%%%%%%%%%%%%%%%%%%%%%%%%%%%%%%%%%%%%%%%%%%%%%

\begin{lema}\label{chave6}
Let $\frac{\sqrt{2}}{\sqrt{\beta}}\leq T \leq \frac{1}{2}$ and $\beta \geq 8$.
Then,
\begin{align}
\nor{\Psi_T(\cdot)\,I_a(\cdot)}{H_t^{1/2}} &\leq
C\,\alpha\,(\beta +\alpha^2)\,T^{1/2}\,\nor{f}{H^{-1/2}}, \quad \text{if} \quad a<-\alpha,
\label{thankgodone} \\
\intertext{and,}
\nor{\Psi_T(\cdot)\,I_a(\cdot)}{H_t^{1/2}} &\leq
C\,\alpha\,(\beta+\alpha^2)\,e^{2\alpha}\,T^{1/2}\,\nor{f}{H^{-1/2}}, \quad \text{if} \quad a\leq \alpha.
\label{sothankgodone}
\end{align}
\end{lema}
\begin{proof}[Proof]
(\ref{chave7}) and (\ref{salvador}) give (\ref{thankgodone}), and, (\ref{sothankgodone}) is consequence of (\ref{thankgodone}) and (\ref{thankgod}).
\end{proof}

%%%%%%%%%%%%%%%%%%%%%%%%%%%%%%%%%%%%%%%%%%%%%%%%%%%%%%%%%%%%%%%%%%%%%%%%%%%%%%%%%%%%%%%%%%%%%%%%%%%%%%%%%%%%%%%%%%%%%%%%%%%%%%%%%%%%%%%%%%%%%%%%%%%%%%%
%%%%%%%%%%%%%%%%%%%%%%%%%%%%%%%%%%%%%%%%%%%%%%%%%%%%%%%%%%%%%%%%%%%%%%%%%%%%%%%%%%%%%%%%%%%%%%%%%%%%%%%%%%%%%%%%%%%%%%%%%%%%%%%%%%%%%%%%%%%%%%%%%%%%%%%
%%%%%%%%%%%%%%%%%%%%%%%%%%%%%%%%%%%%%%%%%%%%%%%%%%%%%%%%%%%%%%%%%%%%%%%%%%%%%%%%%%%%%%%%%%%%%%%%%%%%%%%%%%%%%%%%%%%%%%%%%%%%%%%%%%%%%%%%%%%%%%%%%%%%%%%

\begin{proof}[Proof of Proposition \ref{forcingterm}]
From the definition of the $\mathcal{Y}_{s,b}$ norm, we have
\begin{align}
&\nora{\Psi_T(t)\int_0^tV_{\lambda}(t-t')F(t')\,dt'}{\mathcal{Y}_{s,1/2}}{2} \notag \\
%&=\sum_{k\neq 0}\int_{-\infty}^{\infty} (1+|\tau-k^3|)\,|k|^{2s} \Bigl|\Bigl(\Psi_T(t)\int_0^t  V_{\eta}(t-t')F(t')\,dt'\Bigr)^{\wedge}(k,\tau)\Bigr|^2\,d\tau \notag \\
%&=\sum_{k\neq 0}\int_{-\infty}^{\infty} (1+|\tau|)\,|k|^{2s} \Bigl|\Bigl(\Psi_T(t)\int_0^tV_{\eta}(t-t')F(t')\,dt'\Bigr)^{\wedge}(k,\tau +k^3)\Bigr|^2\,d\tau \notag \\
%&=\sum_{k\neq 0}\int_{-\infty}^{\infty} (1+|\tau|)\,|k|^{2s} \Bigl|\Bigl(e^{-itk^3}\Psi_T(t)\int_0^tV_{\eta}(t-t')F(t')\,dt'\Bigr)^{\wedge}(k,\tau)\Bigr|^2\,d\tau \notag \\
&=\sum_{k\neq 0}\langle k\rangle^{2s} \int_{-\infty}^{\infty} (1+|\tau|) \Bigl|\Bigl(e^{-itk^3}\Psi_T(t)\int_0^te^{ik^3(t-t')+\eta|t-t'|\Phi(k)}\widehat{F}(k,t')\,dt'\Bigr)^{\wedge}(\tau)\Bigr|^2\,d\tau \notag \\
%&=\sum_{k\neq 0} |k|^{2s} \nora{\Psi_T(t)\int_0^t e^{-ik^3t'+\eta|t-t'|\Phi(k)}\widehat{F}(k,t')\,dt'}{H^{1/2}_t}{2} \notag \\
&=\nora{\langle k \rangle^s \nor{\Psi_T(t)\int_0^t e^{\eta|t-t'|\Phi(k)}\,[e^{-ik^3t'}\widehat{F}(k,t')]\,dt'}{H^{1/2}_t}}{l^2_k}{2} \notag \\
&\leq \nora{\langle k\rangle^s\,C\,\eta\,\alpha\,(\beta+(\eta \,\alpha)^2)\,e^{2\alpha\,\eta}\, T^{1/2}\, \nor{ e^{-ik^3t}\widehat{F}(k,t)}{H^{-1/2}_t}}{l^2_k}{2} \label{aplicalema} \\
&= C^2\,(\eta\,\alpha)^2\,(\beta+(\eta\,\alpha)^2)^2\,e^{4\eta\,\alpha}\, T\, \nora{\langle k\rangle^s \langle\tau -k^3\rangle^{-1/2}\widehat{F}(k,\tau)}{l_k^2L_{\tau}^2}{2} .\label{cotapartenolin}
%&= C^2\,(\eta\,\alpha )^2\,(\beta+(\eta \,\alpha)^2)^2\,e^{4\eta \,\alpha}\, T\, \nora{F}{\mathcal{Y}_{s,-1/2}}{2}
\end{align}
In the inequality (\ref{aplicalema}) we apply the Lemma \ref{chave6}. (\ref{cotapartenolin}) implies (\ref{cotapartenaolineal}).
\end{proof}

%%%%%%%%%%%%%%%%%%%%%%%%%%%%%%%%%%%%%%%%%%%%%%%%%%%%%%%%%%%%%%%%%%%%%%%%%%%%%%%%%%%%%%%%%%%%%%%%%%%%%%%%%%%%%%%%%%%%%%%%%%%%%%%%%%%%%%%%%%%%%%%%%%%%%%%%%
%%%%%%%%%%%%%%%%%%%%%%%%%%%%%%%%%%%%%%%%%%%%%%%%%%%%%%%%%%%%%%%%%%%%%%%%%%%%%%%%%%%%%%%%%%%%%%%%%%%%%%%%%%%%%%%%%%%%%%%%%%%%%%%%%%%%%%%%%%%%%%%%%%%%%%%%%
%%%%%%%%%%%%%%%%%%%%%%%%%%%%%%%%%%%%%%%%%%%%%%%%%%%%%%%%%%%%%%%%%%%%%%%%%%%%%%%%%%%%%%%%%%%%%%%%%%%%%%%%%%%%%%%%%%%%%%%%%%%%%%%%%%%%%%%%%%%%%%%%%%%%%%%%%
%%%%%%%%%%%%%%%%%%%%%%%%%%%%%%%%%%%%%%%%%%%%%%%%%%%%%%%%%%%%%%%%%%%%%%%%%%%%%%%%%%%%%%%%%%%%%%%%%%%%%%%%%%%%%%%%%%%%%%%%%%%%%%%%%%%%%%%%%%%%%%%%%%%%%%%%%
%%%%%%%%%%%%%%%%%%%%%%%%%%%%%%%%%%%%%%%%%%%%%%%%%%%%%%%%%%%%%%%%%%%%%%%%%%%%%%%%%%%%%%%%%%%%%%%%%%%%%%%%%%%%%%%%%%%%%%%%%%%%%%%%%%%%%%%%%%%%%%%%%%%%%%%%%
%%%%%%%%%%%%%%%%%%%%%%%%%%%%%%%%%%%%%%%%%%%%%%%%%%%%%%%%%%%%%%%%%%%%%%%%%%%%%%%%%%%%%%%%%%%%%%%%%%%%%%%%%%%%%%%%%%%%%%%%%%%%%%%%%%%%%%%%%%%%%%%%%%%%%%%%%

\setcounter{equation}{0}
\section{\textsc{Local Well-posedness in $H^s(\mathbb{T})$}}

Consider the $\lambda$-periodic initial value problem (\ref{fivp}) with periodic initial data $u_0\in H^s_{\lambda}$, $s\geq -1/2$. We show first that, for arbitrary $\lambda$, this problem is well-posed on a time interval of size $\sim 1$ provided $\nor{u_0}{H^{-1/2}_{\lambda}}$ is sufficiently small. Then we show by a rescaling argument that (\ref{fivp}) is locally well-posed for arbitrary initial data $u_0\in H^s_{\lambda}$. As mentioned before in Remark \ref{meanzero}, we restrict our attention to initial data having zero $x$-mean.

\begin{proof}[Proof of the Theorem \ref{maintheorem}] Fix $u_0\in H^s_{\lambda}$, $s\geq -1/2$ and for $w\in Z^{-1/2}$ define
\begin{align}
(\mathcal{A}w)(t)=\Psi(t)V_{\lambda }(t)u_0 - \Psi(t)\int_0^tV_{\lambda }(t-t')\,(\Psi(t')\,w(t'))\,dt' . \label{contraction}
\end{align}
The bilinear estimate (\ref{bilestckstt}) shows that $u\in Y^{-1/2}$ implies $\Psi(t)\,\partial_x(u^2)\in Z^{-1/2}$ so the (nonlinear) operator
$$\Gamma(u)=\mathcal{A}\Bigl(\frac{1}{2}\,\partial_x(u^2)\Bigr)$$
is defined on $Y^{-1/2}$. Observe that $\Gamma(u)=u$ is equivalent, at least for $t\in [-1,1]$, to (\ref{integralequation}), which is equivalent to (\ref{fivp}).\\ \\
{\bf Claim 1.} $\Gamma$ : (bounded subsets of $Y^{-1/2}$)$\longmapsto$ (bounded subsets of $Y^{-1/2}$).\\
Since
$$\Gamma(u)=\Psi(t)V_{\lambda }(t)u_0 - \Psi(t)\int_0^tV_{\lambda }(t-t')\,\Bigl(\frac{\Psi(t')}{2}\,\partial_xu^2(t')\Bigr)\,dt',$$
we estimate using (\ref{linestuno}), (\ref{linestdos}) and the bilinear estimate (\ref{bilestckstt}):
\begin{align}
\nor{\Gamma(u)}{Y^{-1/2}}&\leq \nor{\Psi(t)\,V_{\lambda }(t)u_0}{Y^{-1/2}}+\nor{\Psi(t)\int_0^tV_{\lambda }(t-t')\,(\frac{\Psi(t')}{2}\,\partial_xu^2(t'))\,dt'}{Y^{-1/2}} \notag \\
&\leq C_1\,\nor{u_0}{H^s_{\lambda}} + C_2\,\nor{\Psi(t)\,\partial_xu^2}{Z^{-1/2}} \notag \\
&\leq C_1\,\nor{u_0}{H^s_{\lambda}} + C_2\,C_3\,\lambda^{0+}\,\nora{u}{Y^{-1/2}}{2} \label{claimuno}
\end{align}
and the claim is proved. \\ \\
Now, we consider the ball
$$\mathfrak{B}=\{w\in Y^{-1/2}: \nor{w}{Y^{-1/2}}\leq C_4\,\nor{u_0}{H^{-1/2}_{\lambda}}\}.$$
{\bf Claim 2.} $\Gamma$ is a contraction on $\mathfrak{B}$ if $\nor{u_0}{H^{-1/2}_{\lambda}}$ is sufficiently small. \\
We wish to prove that for some $\theta \in (0,1)$,
$$\nor{\Gamma(u)-\Gamma(v)}{Y^{-1/2}}\leq \theta \, \nor{u-v}{Y^{-1/2}}$$
for all $u$, $v \in \mathfrak{B}$. Since $u^2-v^2=(u+v)(u-v)$, we can see that
\begin{align}
\nor{\Gamma(u)-\Gamma(v)}{Y^{-1/2}}&\leq \nor{-\Psi(t)\int_0^tV_{\lambda }(t-t')\,\frac{\Psi(t')}{2}\,\partial_x(u^2-v^2)(t')\,dt'}{Y^{-1/2}} \notag \\
&\leq C_2\,\nor{\Psi(t)\,\partial_x(u+v)(u-v)}{Z^{-1/2}} \notag \\
&\leq C_2\,C_3\,\lambda^{0+}\,(\nor{u}{Y^{-1/2}}+\nor{u}{Y^{-1/2}})\,\nor{u-v}{Y^{-1/2}} \notag \\
&\leq 2\,C_2\,C_3\,C_4\,\lambda^{0+}\,\nor{u_0}{H^{-1/2}_{\lambda}}\,\nor{u-v}{Y^{-1/2}}. \label{claimdos}
\end{align}
(\ref{claimdos}) holds because $u$, $v \in \mathfrak{B}$. Hence, for fixed $\lambda$, if we take $\nor{u_0}{H^{-1/2}_{\lambda}}$ so small such that
\begin{equation}
2\,C_2\,C_3\,C_4\,\lambda^{0+}\,\nor{u_0}{H^{-1/2}_{\lambda}}\ll 1  \label{smallcondition}
\end{equation}
the contraction estimate is verified.\\ \\
The preceding discussion establishes well-posedness of (\ref{fivp}) on a $O(1)$-sized time interval for any initial data satisfying (\ref{smallcondition}). To prove that our result holds for every given data $u_0$ in $H^s_{\lambda}$ and not only for small data as (\ref{smallcondition}), let us perform the following scale change
\begin{align}
v(x,t)&=\frac{1}{\sigma^2}u\Bigl(\frac{x}{\sigma},\frac{t}{\sigma^3}\Bigr)\label{changescale} \\
\intertext{where $\sigma\geq \alpha$. So, $v$ is periodic with respect to the $x$ variable with period $\sigma \lambda$, hence for $k\in \mathbb{Z}/\sigma \lambda$}
\widehat{v}(k)&=\frac{1}{\sigma^2}\,\int_0^{\sigma \lambda}e^{-2\pi ikx}\,u(x/\sigma)\,dx=\frac{1}{\sigma}\,\widehat{u}(k\sigma),
\end{align}
and $v$ satisfies the equation
\begin{align*}
\sigma^3v_t(x,t)+\sigma^3v_{xxx}(x,t)+\sigma^3vv_x(x,t)+\eta Sv(x,t)=0,
\end{align*}
where the operator $S$ is defined by $(Sv)^{\wedge}(k):=-\Phi(\sigma k) \hat{v}(k)$ and so,
$ Sv(x,t)=\frac{1}{\sigma^2} Lu\Bigl(\dfrac{x}{\sigma} , \dfrac{t}{ \sigma^3}\Bigr)$.
Hence, $v_t + v_{xxx} + vv_x + \eta \frac{1}{\sigma^3} Sv = 0$. Considering $\widetilde{S}=\dfrac{1}{\sigma^3}S$, $v$ satisfies
\begin{equation}\label{newfivp}
\left\{
\begin{aligned}
v_t+v_{xxx}+vv_x+ \eta \widetilde{S} v &= 0 \qquad x\in [0,\sigma \,\lambda], \;\; t\in (0,+\infty) \\
v(x,0)&=v_0(x)= \dfrac{1}{\sigma^2}\,u_0\Bigl(\dfrac{x}{\sigma}\Bigr),
\end{aligned}
\right.
\end{equation}
where
$$(\widetilde{S}v)^{\wedge}(k)=\dfrac{1}{\sigma^3}(S v)^{\wedge}(k)= -\dfrac{1}{\sigma^3}\,\Phi(\sigma k)\hat{v}(k)\quad \text{and}\quad \dfrac{1}{\sigma^3}\,\Phi(\sigma k)\leq \dfrac{\alpha}{\sigma^3}\leq 1.$$
Finally, consider (\ref{fivp}) with $\lambda=\lambda_0$ fixed and $u_0\in H^s_{\lambda_0}$, $s\geq -1/2$. This problem is well-posed on a small time interval $[0,\delta]$ if and only if the $\sigma$-rescaled problem (\ref{newfivp}) is well-posed on $[0,\sigma^3\,\delta]$. A calculation shows that
$$\nor{v_0}{H^{-1/2}_{\sigma\,\lambda_0}}\leq \frac{1}{\sigma}\,\nor{u_0}{H^{-1/2}_{\lambda_0}}.$$
Observe that
$$(\sigma\,\lambda_0)^{0+}\,\nor{v_0}{H^{-1/2}_{\sigma\,\lambda_0}}\leq \frac{(\sigma\,\lambda_0)^{0+}}{\sigma}\,\nor{u_0}{H^{-1/2}_{\lambda_0}}\ll 1,$$
provided $\sigma=\sigma(\lambda_0, \nor{u_0}{H^{-1/2}_{\lambda_0}})$ is taken to be sufficiently large. This verifies (\ref{smallcondition}) for the problem (\ref{newfivp}) proving well-posedness of (\ref{newfivp}) on the time interval, say $[0,1]$. Hence, (\ref{fivp}) is locally well-posed for $t\in [0,\sigma^{-3}]$.
\end{proof}

%%%%%%%%%%%%%%%%%%%%%%%%%%%%%%%%%%%%%%%%%%%%%%%%%%%%%%%%%%%%%%%%%%%%%%%%%%%%%%%%%%%%%%%%%%%%%%%%%%%%%%%%%%%%%%%%%%%%%%%%%%%%%%%%%%%%%%%%%%%%%%%%%%%%%%%%
%%%%%%%%%%%%%%%%%%%%%%%%%%%%%%%%%%%%%%%%%%%%%%%%%%%%%%%%%%%%%%%%%%%%%%%%%%%%%%%%%%%%%%%%%%%%%%%%%%%%%%%%%%%%%%%%%%%%%%%%%%%%%%%%%%%%%%%%%%%%%%%%%%%%%%%%
%%%%%%%%%%%%%%%%%%%%%%%%%%%%%%%%%%%%%%%%%%%%%%%%%%%%%%%%%%%%%%%%%%%%%%%%%%%%%%%%%%%%%%%%%%%%%%%%%%%%%%%%%%%%%%%%%%%%%%%%%%%%%%%%%%%%%%%%%%%%%%%%%%%%%%%%
%%%%%%%%%%%%%%%%%%%%%%%%%%%%%%%%%%%%%%%%%%%%%%%%%%%%%%%%%%%%%%%%%%%%%%%%%%%%%%%%%%%%%%%%%%%%%%%%%%%%%%%%%%%%%%%%%%%%%%%%%%%%%%%%%%%%%%%%%%%%%%%%%%%%%%%%
%%%%%%%%%%%%%%%%%%%%%%%%%%%%%%%%%%%%%%%%%%%%%%%%%%%%%%%%%%%%%%%%%%%%%%%%%%%%%%%%%%%%%%%%%%%%%%%%%%%%%%%%%%%%%%%%%%%%%%%%%%%%%%%%%%%%%%%%%%%%%%%%%%%%%%%%
%%%%%%%%%%%%%%%%%%%%%%%%%%%%%%%%%%%%%%%%%%%%%%%%%%%%%%%%%%%%%%%%%%%%%%%%%%%%%%%%%%%%%%%%%%%%%%%%%%%%%%%%%%%%%%%%%%%%%%%%%%%%%%%%%%%%%%%%%%%%%%%%%%%%%%%%

\renewcommand{\refname}

%%%%%%%%%%%%%%%%%%%%%%%%%%%%%%%
%%%%%%%%%%%%%%%%%%%%%%%%%%%%%%%

\end{document}